\documentclass[11pt]{amsart}
\usepackage{amscd,amsmath,amssymb,amsfonts,verbatim}
\usepackage[cmtip, all]{xy}
\usepackage{MnSymbol}
\usepackage{comment}

\newtheorem{thm}{Theorem}[section]
\newtheorem{prop}[thm]{Proposition}
\newtheorem{lem}[thm]{Lemma}
\newtheorem{cor}[thm]{Corollary}

\theoremstyle{definition}

\newtheorem{defn}[thm]{Definition}
\theoremstyle{remark}
\newtheorem{remk}[thm]{Remark}
\newtheorem{remks}[thm]{Remarks}

\newtheorem{exm}[thm]{Example}
\newtheorem{exms}[thm]{Examples}
\newtheorem{notat}[thm]{Notation}
\numberwithin{equation}{section}

{\hfill$\square$\end{defn}}
\newenvironment{rem}{\begin{remk}}%
{\hfill$\square$\end{remk}}
{\hfill$\square$\end{remks}}
{\hfill$\square$\end{exm}}
{\hfill$\square$\end{exms}}
{\hfill$\square$\end{notat}}

\newcommand{\thmref}{Theorem~\ref}
\newcommand{\propref}{Proposition~\ref}
\newcommand{\corref}{Corollary~\ref}

\newcommand{\lemref}{Lemma~\ref}

\newcommand{\sC}{{\mathcal C}}

\newcommand{\sE}{{\mathcal E}}
\newcommand{\sF}{{\mathcal F}}

\newcommand{\sI}{{\mathcal I}}

\newcommand{\sK}{{\mathcal K}}

\newcommand{\sO}{{\mathcal O}}

\newcommand{\A}{{\mathbb A}}

\newcommand{\F}{{\mathbb F}}

\newcommand{\N}{{\mathbb N}}
\renewcommand{\P}{{\mathbb P}}
\newcommand{\Q}{{\mathbb Q}}

\newcommand{\W}{{\mathbb W}}

\newcommand{\Z}{{\mathbb Z}}

\newcommand{\fm}{{\mathfrak m}}
\newcommand{\fn}{{\mathfrak n}}
\newcommand{\fp}{{\mathfrak p}}

\newcommand{\CH}{{\rm CH}}

\newcommand{\surj}{\twoheadrightarrow}
\newcommand{\inj}{\hookrightarrow}

\newcommand{\Hom}{{\rm Hom}}

\newcommand{\Spec}{{\rm Spec \,}}

\newcommand{\Sch}{{\operatorname{\mathbf{Sch}}}}

\newcommand{\Sm}{{\mathbf{Sm}}}

\newcommand{\cyc}{{\operatorname{\rm cyc}}}

\newcommand{\ds}{{/\kern-3pt/}}

\newcommand{\un}{\underline}
\newcommand{\ov}{\overline}

\renewcommand{\dim}{\text{\rm dim}}

\newcommand{\tuborg}{\left\{\begin{array}{ll}}
\newcommand{\sluttuborg}{\end{array}\right.}

\newcommand{\sfs}{{\rm sfs}}

\newcommand{\Tz}{{\rm Tz}}
\newcommand{\dlog}{{\rm dlog}}

\newcommand{\TCH}{{\rm TCH}}
\newcommand{\wt}{\widetilde}
\newcommand{\wh}{\widehat}

\newcounter{elno}

\newcounter{elno-abc}   

\newcounter{elno-abc-prime}

\begin{document}
\title{Relative $K$-theory via 0-cycles in finite
characteristic}
\author{Rahul Gupta, Amalendu Krishna}
\address{Fakult\"at f\"ur Mathematik, Universit\"at Regensburg, 
93040, Regensburg, Germany.}
\email{Rahul.Gupta@mathematik.uni-regensburg.de}
\address{Department of Mathematics, Indian Institute of Science, Bangalore, 560012, India.}
\email{amalenduk@iisc.ac.in}

%\dedicatory{}

\keywords{Algebraic cycles, additive higher Chow groups,
relative $K$-theory}        

\subjclass[2010]{Primary 14C25; Secondary 19E08, 19E15}

\maketitle

\begin{quote}\emph{Abstract.}  
Let $R$ be a regular semi-local ring, essentially of finite type over an infinite
perfect field of characteristic $p > 0$.
We show that the known cycle class map from the Chow group of 0-cycles with 
modulus to the relative $K$-theory 
induces a pro-isomorphism between the additive higher Chow groups of
relative 0-cycles and the relative $K$-theory of truncated polynomial
rings over $R$. This settles the problem of
completely describing the relative $K$-theory of such rings
via the cycle class map.
\end{quote}
%\end{abstract}
\setcounter{tocdepth}{1}
%\maketitle
\tableofcontents  

%\maketitle
%\tableofcontent

\section{Introduction}\label{sec:Intro}
Ever since the invention of additive Chow groups and higher Chow groups 
with modulus, it has been an open question whether these groups
together would give rise to a motivic cohomology which could describe the
algebraic $K$-theory of non-reduced schemes. Existence of such a
motivic cohomology was conjectured in the seminal work of
Bloch and Esnault \cite{BE2}.

There are usually two ways to solve this question; either construct a direct
cycle class map from the Chow groups with modulus to relative $K$-theory,
or, construct an Atiyah-Hirzebruch type spectral sequence.
For smooth schemes, both approaches have been shown to be successful
in describing the algebraic $K$-theory in terms of algebraic cycles.  
However, this question remains unsolved for singular schemes.

In \cite{Levine-1}, Levine constructed cycle class maps with rational
coefficients from Bloch's higher Chow groups \cite{Bloch-1} to the 
algebraic $K$-groups of a smooth scheme over a field. He showed that these maps 
induce isomorphisms between the higher Chow groups and
the Adams graded pieces of the algebraic $K$-groups of the scheme. 

Motivated by Levine's work, the authors constructed in \cite{GK} an explicit cycle
class map (with integral coefficients) from the higher Chow groups
of 0-cycles with modulus to the relative $K$-theory in the setting of 
pro-abelian groups. The main result of \cite{GK} was that this 
cycle class map induces a pro-isomorphism between the additive 
higher Chow groups of
relative 0-cycles and relative $K$-theory of truncated polynomial
rings over a regular semi-local
ring, essentially of finite type over a characteristic zero field.
The goal of this manuscript is to extend this result to
positive characteristic.
%This completes the program of describing the relative 
%$K$-theory of truncated polynomial rings over regular semi-local rings
%containing an infinite field by algebraic cycles via an explicit cycle class map.

To state our main result, recall from \cite{BS} and \cite{KP} that
for a smooth scheme $X$ of dimension $d$ which is essentially of finite
type over a field $k$ and an effective Cartier divisor $D \subset X$,
the higher Chow groups of codimension $q$-cycles with modulus are denoted by
$\CH^{q}(X|D; n)$. Let $K(X,D)$ denote the relative $K$-theory spectrum.
In order to study the relative algebraic $K$-theory in terms of 0-cycles with
modulus, it was shown in \cite{GK} that 
there exists a cycle class map 
\begin{equation}\label{eqn:CCM}
cyc_{X|D} \colon \left\{\CH^{n+d}(X|m D; n)\right\}_{m \ge 1} \to \left\{K_n(X,mD)\right\}_{m \ge 1}
\end{equation}  
in the setting of pro-abelian groups.  
This cycle class map coincides with that of Levine when $D = \emptyset$.

Recall now that for an equi-dimensional 
scheme $X$ over $k$, the Chow group with modulus
$\CH^q(X \times \A^1_k| X \times {(m+1)}\{0\},n)$ 
is the same thing as the additive higher Chow group 
of codimension $q$-cycles $\TCH^q(X,n+1;m)$ (see \cite{KP-1}). 
Applying ~\eqref{eqn:CCM} to $X = \Spec(k)$
and using the natural connecting isomorphism 
$\partial \colon K_{n+1}({k[x]}/{(x^{m})}, (x)) \xrightarrow{\cong} 
K_n(\A^1_k, m \{0\})$, we see that ~\eqref{eqn:CCM} is the same thing as the map
\begin{equation}\label{eqn:CCM-k}
cyc_{k} \colon \left\{\TCH^{n+1}(k, n+1;m)\right\}_{m \ge 1} \to 
\left\{K_{n+1}({k[x]}/{(x^{m})}, (x))\right\}_{m \ge 1}.
\end{equation}  

The main property of $cyc_{k} $ is that its definition is very explicit on the set of 
generators $\Tz^{n+1}(k, n+1;m)$ (see \S~\ref{sec:CCM}).
% such that if $z \in \A^1_k \times \sq^n(k)$ is a $k$-rational point, then $cyc_{k} ([Z]) =\partial^{-1}(\{y_1(z), \dots, y_n(z)\})$, where $y_1, \dots, y_n$ are co-ordinates of $\sq^n$.
This property of $cyc_k$ often turns out to be very
useful in the study of $K$-theory via algebraic cycles.
If $R$ is, more generally, a regular semi-local ring containing $k$, the map $cyc_k$ directly extends
to an explicit cycle class map 
\begin{equation}\label{eqn:CCM-R}
cyc_R \colon \Tz^{n+1} _{\sfs}(R, n+1;m) \to  K_{n+1}(R[x]/(x^{m+1}), (x)),
\end{equation}
 where 
$ \Tz^{n+1} _{\sfs}(R, n+1;m) \subset  \Tz^{n+1}(R, n+1;m)$ is a subgroup of `sfs' cycles (see \S~\ref{sec:Rel-cyc} for the definition of sfs cycles and \S~\ref{sec:CCM-R-0}
for the definition of $cyc_R$). 
%If $R$ a regular semi-local ring which is essentially of finite type over an infinite field $k$, then the natural map $\Tz^{n+1} _{\sfs}(R, n+1;m) \surj \TCH^{n+1} (R, n+1;m)$ is surjective by \cite{KP-3}.

The initial motivation behind the discovery of additive higher Chow groups 
by Bloch and Esnault \cite{BE2} was to know if a cycle class map such as $cyc_R$
passes through rational equivalence and if the resulting map is an isomorphism.
The main objective of this paper is to provide the following partial answer to the
question of Bloch and Esnault.
In fact, we construct a new cycle class map in the pro-setting and show that it is an isomorphism when
$k$ is any perfect field. We subsequently show that this new cycle class map coincides with
$cyc_R$ of ~\eqref{eqn:CCM-R} when $k$ is furthermore infinite.

\begin{thm}\label{thm:Main-0}
Let $R$ be a regular semi-local ring which is essentially of finite type over a field $k$ 
such that ${\rm char}(k) > 0$. 
%Assume that either $k$ is infinite or $R$ is local. 
Let $n \ge 1$ be an integer. Then there exists a cycle class map
\[
cyc'_R \colon \left\{\TCH^{n}(R, n;m)\right\}_{m \ge 1} \to 
\left\{K_{n}({R[x]}/{(x^{m})}, (x))\right\}_{m \ge 1}.
\]
This map satisfies the following.
\begin{enumerate}
\item
  $cyc'_R$ is natural  in $R$.
\item
  $cyc'_R$ is injective.
\item
$cyc'_R$ is an isomorphism if $k$ is perfect.
\item
  The composite map
  \[
    \left\{\Tz^{n} _{\sfs}(R, n;m)\right\}_{m \ge 1} \to
    \left\{\TCH^{n}(R, n;m)\right\}_{m \ge 1} \xrightarrow{cyc'_R}
    \left\{K_{n}({R[x]}/{(x^{m})}, (x))\right\}_{m \ge 1}
  \]
  coincides with $cyc_R$ if $R = k$ or $k$ is infinite. 
\end{enumerate}
\end{thm}

Using \cite{GK} in characteristic zero and \thmref{thm:sfs-lemma} otherwise, we obtain the
following.

\begin{cor}\label{cor:Main-0-0}
  Let $R$ be a regular semi-local ring which is essentially of finite type over an infinite
  perfect field. Then $cyc_R$ induces an isomorphism
  \[
 cyc_R \colon \left\{\TCH^{n}(R, n;m)\right\}_{m \ge 1} \to 
\left\{K_{n}({R[x]}/{(x^{m})}, (x))\right\}_{m \ge 1}.   
\]
\end{cor}

We remark that the only hurdle in extending the above corollary to finite base fields is the
lack of sfs-moving lemma. The proof of this moving lemma given in \cite{KP-3} breaks down
if the base field is finite. However, we expect that the new Bertini theorems of 
\cite{Ghosh-Krishna} may be enough to prove the sfs-moving lemma over finite base fields.

\vskip .3cm

\begin{comment}
The results of \cite{GK} and \thmref{thm:Main-0} together completely settle the
question whether the relative $K$-theory of truncated polynomial
rings can be described in terms of additive Chow groups of relative 0-cycles.
In particular, we see using (2) and (4) that the map 
$cyc_k$ of ~\eqref{eqn:CCM-k} is an isomorphism if $k$ is perfect.

We should caution the reader that if $R$ is not a field, then the existence of $cyc_R$ does not follow from
~\eqref{eqn:CCM} because in this case $\TCH^n(R,n;m)$ is not the 0-cycle group. Its elements are cycles
having same dimension as that of $R$. So one needs to construct $cyc_R$ separately. The group $\TCH^{n+d}(R,n;m)$ is actually zero if $\dim(R) = d > 0$,
a consequence of the sfs-moving lemma of \cite{KP-3}.
\end{comment}

For a semi-local ring $R$, let $K^M_*(R)$ denote the Milnor $K$-theory of $R$. 
When $R$ has a finite residue field, 
$K^M_*(R)$ is taken to be the one defined by Gabber (unpublished) and
Kerz \cite{Kerz10}.
Let $K^M_n({R[x]}/{(x^{m})}, (x))$ denote the kernel of the
canonical restriction map $K^M_n({R[x]}/{(x^{m})}) \to K^M_n(R)$.
Unless $R$ is local, there may not exist a natural map
from the Milnor $K$-theory ({\`a} la Gabber-Kerz) to the Quillen $K$-theory of
${R[x]}/{(x^{m})}$. Nonetheless, we show in this manuscript 
(see \corref{cor:Milnor-Quillen-map}) that
there is indeed a natural map
of pro-abelian groups
\[
\psi_R \colon  \left\{{K}^M_n({R[x]}/{(x^{m})})\right\}_{m \ge 1} \to 
 \left\{K_n({R[x]}/{(x^{m})})\right\}_{m \ge 1}.
\]
This induces a natural map between the pro-relative $K$-groups as well.
Furthermore, this restricts to the known canonical map when we
replace $R$ by any of its localizations.
The main step in the proof of \thmref{thm:Main-0} (except its last part) is the following
extension of \cite[Theorem~1.3(1)]{GK} to positive characteristic.

\begin{thm}\label{thm:Main-1}
Let $R$ be a regular semi-local ring which is essentially of finite type over a field of characteristic $p > 0$.
Let $n \ge 1$ be an integer. Then there exists a cycle class map
\[
cyc^M_R \colon \left\{\TCH^{n}(R, n;m)\right\}_{m \ge 1} \to 
\left\{K^M_{n}({R[x]}/{(x^{m})}, (x))\right\}_{m \ge 1}
\]
which is natural in $R$ and is an isomorphism.
\end{thm}

 The cycle class map $cyc'_R$ in \thmref{thm:Main-0} is, by definition, the 
 composition $\psi_R \circ cyc^M_R$.
We remark that by a result of Morrow \cite{Morrow} (which implicitly
uses \thmref{thm:Milnor-final} of this paper, see the proof of
\propref{prop:M-Q-pro}), one knows that the canonical map
from the relative Milnor $K$-theory to the relative Quillen $K$-theory
is a pro-isomorphism when $R$ is local. However, there are two points
to be noted regarding \thmref{thm:Main-1}. First, the pro-isomorphism
between the relative Milnor $K$-theory and the Quillen $K$-theory for
semi-local rings is not a straightforward deduction from the local
case.
Second, and more important, \thmref{thm:Main-1}, along with 
\thmref{thm:Main-0}(4), asserts that
the cycle class map $cyc_R$ (whose study is our main interest)
from the additive higher Chow groups to the
relative Quillen $K$-theory in ~\eqref{eqn:CCM-R}
factors through the relative Milnor $K$-theory if the base field if infinite and perfect.
This plays a very important role in understanding the cycle class
map and in our proof. A similar result in characteristic zero was proven in
\cite{GK}.

\vskip .2cm

A fundamental fact 
in Voevodsky's theory of motivic cohomology is that 
if $R$ is an equi-characteristic regular semi-local ring, then
its motivic cohomology in the equal bi-degree (the Milnor range) 
is isomorphic to the Milnor $K$-theory of $R$
(see \cite{EM}, \cite{NS} and \cite{Totaro}).
\thmref{thm:Main-1} (3) says that this isomorphism
also holds for truncated polynomial algebras over such rings.
This provides a key evidence that if one could extend 
Voevodsky's theory of motives to so-called
fat points (infinitesimal extensions of spectra of fields), then
the underlying motivic cohomology groups must coincide with the additive higher 
Chow groups, at least in the setting of pro-abelian groups (see \cite{KP-5}).

\vskip .3cm

\subsection{A brief outline of the proofs}\label{sec:Outline}
Our strategy for proving \thmref{thm:Main-0} (except part (4)) is to define the cycle class map 
$cyc_R$ as the composition of $\psi_R$ with a cycle class map $cyc^M_R$, to 
the relative Milnor $K$-theory of the truncated polynomial ring over the underlying regular
semi-local ring $R$. The two known results that are used to achieve this are
the `Chow-Witt isomorphism theorems' of R{\"u}lling \cite{R} and Krishna-Park \cite{KP-2} and the 
`Milnor-Witt isomorphism theorem' of R{\"u}lling-Saito \cite{RS}.
%(see \thmref{thm:RS-main}).
But these two results are not quite enough. 

Going beyond, a fundamental result of independent interest
that we need to prove is that there is a 
pro-isomorphism between the Milnor $K$-theories of R{\"ulling}-Saito \cite{RS},
Kato-Saito \cite{Kato-Saito} and Gabber-Kerz \cite{Kerz10}
(see \thmref{thm:Milnor-final}).
%(see \corref{cor:Milnor-truncated}).
This allows us to define a cycle class map from the additive higher Chow groups to the relative
Milnor $K$-theory of the truncated polynomial ring (see \S~\ref{sec:CCM-MK}).
This map is easily seen to be an isomorphism by its construction.

The next step is to show the existence of a canonical map
from the relative Milnor $K$-theory to relative Quillen $K$-theory
of truncated polynomial rings over a regular semi-local ring
(see Corollary~\ref{cor:Milnor-Quillen-map}). This requires 
care if the base field is finite and the ring is not local.
The third step is show that 
the
above described map 
%from the relative Milnor $K$-theory to the relative Quillen $K$-theory of 
%truncated polynomial rings 
is a pro-isomorphism. This is done in
\propref{prop:M-Q-pro} using \thmref{thm:Milnor-final} and a result of Morrow 
\cite[Corollary~5.5]{Morrow}.

The final step to explicitly describe this composition when $R$  is a field or is 
essentially of finite type over an infinite field. Under these assumptions, 
we define an explicit  cycle class map to the relative Quillen $K$-theory of 
truncated polynomials in \S~\ref{sec:CCM} and 
\S~\ref{sec:CCM-R*}. 
The proof of \thmref{thm:Main-0} is then completed in 
 \lemref{lem:Gen-commute}. 
%This is proven by using  a result of Hyodo-Kato \cite{HK} (see \lemref{lem:HK-RS}).
%This essentially finishes the proof.

In all the above steps, we have to pay special care if the underlying
regular semi-local ring is not local.
This is because the Gabber-Kerz Milnor $K$-theory is not very well
behaved for such rings. This forces us to work with
their sheafified versions. But this brings in another technical problem.
Namely, a local isomorphism between two pro-sheaves does not 
necessarily imply an isomorphism between them. This is in contrast to the case of sheaves.
To take care of this problem, we always have to
prove a stronger assertion  than merely an isomorphism whenever
we work with pro-abelian groups of $K$-theories and de Rham-Witt forms
on a local ring (see \lemref{lem:pro-sheaves} for a precise version we
need to prove).

\section{The cycle class map}\label{sec:Prelim}
In this section, we shall recall our
main object of study, the cycle class map for 
the additive higher Chow groups of 0-cycles, from \cite{GK}.
Before this, we shall fix our notations, recall the basic
definitions and prove some intermediate results to be used in the proofs
of the main theorems. 

\subsection{Notations}\label{sec:Notations}
%Throughout this manuscript, we fix a base field $k$ of
%exponential characteristic $p \ge 1$. We do not put
%any restriction on $p$ in this section. We shall
%be more specific about it whenever it is required in any given section.
Given a field $k$, we let $\Sch_k$ denote the category of separated
finite type schemes over $k$ and let $\Sm_k$ denote the full subcategory of 
non-singular schemes over $k$.
For $X, Y \in \Sch_k$, we shall denote the product $X \times_k Y$ simply by 
$X \times Y$.
For any point $x \in X$, we shall let $k(x)$ denote the residue field of $x$.
For a reduced scheme $X \in \Sch_k$, we shall let $X^N$ denote the 
normalization of $X$.
%For $q \ge 0$, we shall denote the set of codimension $q$ points of a scheme 
%$X$ by $X^{(q)}$.
Given a closed immersion $D \subset X$
in $\Sch_k$ defined by a sheaf of ideals $\sI_D \subset \sO_X$, 
we shall let $mD \subset X$ denote the closed subscheme
of $X$ defined by the sheaf of ideals $\sI^m_D$ for $m \ge 1$. 
%For an affine scheme $X \in \Sch_k$, we shall let $k[X]$ denote the 
%coordinate ring of $X$.
In this article, we shall always consider the Zariski topology on
a Noetherian scheme whenever we talk about sheaves.

We shall let $\ov{\square}$ denote the projective space $\P^1_k = {\rm Proj}(k[Y_0, Y_1])$ and let
$\square = \ov{\square} \setminus \{1\}$. We shall let $\A^n_k  = \Spec(k[y_1, \ldots , y_n])$ be the
open subset of $\ov{\square}^n$, where $(y_1, \ldots, y_n)$ denotes the coordinate system of 
$\ov{\square}^n$ with $y_j = {Y^j_1}/{Y^j_0}$.  
Given a rational map $f \colon X \dashedrightarrow \ov{\square}^n$ in $\Sch_k$ and a
point $x \in X$ lying in the
domain of definition of $f$, we shall let
$f_i(x) = (y_i \circ f)(x)$, where $y_i \colon \ov{\square}^n \to \ov{\square}$ is the $i$-th projection.
For any $1 \le i \le n$ and $t \in \ov{\square}(k)$, we let $F^t_{n,i}$ denote the closed subscheme of 
$\ov{\square}^n$ given by $\{y_i = t\}$. We let $F^t_n = \stackrel{n}{\underset{i =1}\sum} F^t_{n,i}$.

All rings in this text will be commutative and Noetherian. For such a ring $R$ and an
integer $m \ge 0$, we shall let $R_m = {R[t]}/{(t^{m+1})}$ denote the 
truncated polynomial algebra over $R$.
If in proving a statement in this manuscript, we have to deal with  a ring $R$ and an ideal $I \subset R$,
we shall use the notation $\ov{a}$ for the residue class in $R/I$ of an element $a \in R$. 
If there are several ideals, we shall indicate the quotient in which we consider the residue class
$\ov{a}$.
The tensor product $M \otimes_{\Z} N$ will be denoted simply as $M \otimes N$.
Tensor products over other bases will be explicitly indicated.
For an abelian group $M$ and an integer $n \ge 1$, we let $_n M$ denote the
$n$-torsion subgroup of $M$. 

\subsection{The category of pro-objects}\label{sec:Pro}
Since we shall mostly work in the category of pro-abelian groups, we recall
here the notion of pro-objects in a general category.
By a pro-object in a category $\sC$, we shall mean a sequence of objects $\{A_m\}_{m \ge 0}$
together with a map $\alpha^A_m \colon A_{m+1} \to A_m$ for each $m \ge 0$. We shall write this object
often as $\{A_m\}$. We let ${\rm Pro}(\sC)$ denote the category of pro-objects in $\sC$ with the
morphism set given by
\begin{equation}\label{eqn:Pro-0}
\Hom_{{\rm pro}(\sC)}(\{A_m\}, \{B_m\}) = {\underset{n}\varprojlim} \ 
{\underset{m}\varinjlim}\ \Hom_{\sC}(A_m, B_n).
\end{equation}

It follows that giving a morphism $f \colon \{A_m\} \to \{B_m\}$ 
in ${\rm Pro}(\sC)$ is equivalent to finding a function  
$\lambda \colon \N \to \N$, a map $f_n \colon A_{\lambda(n)} \to B_n$ for each $n \ge 0$ such that 
for each $n' \ge n$, there exists $l \ge \lambda(n), \lambda(n')$ so that the diagram
\begin{equation}\label{eqn:Pro-1}
\xymatrix@C.8pc{
A_l \ar[r] \ar[dr] & A_{\lambda(n')} \ar[r]^-{f_{n'}} & B_{n'} \ar[d] \\
& A_{\lambda(n)} \ar[r]^-{f_n} & B_n}
\end{equation}
is commutative, where the unmarked arrows are the structure maps of $\{A_m\}$ and $\{B_m\}$.
We shall say that $f$ is strict if $\lambda$ is the identity function and $l = \lambda(n') = n' $. 
If $\sC$ admits all sequential limits, we shall denote the limit of $\{A_m\}$ by 
${\underset{m}\varprojlim} \ A_m \in \sC$. If $\sC$ is an abelian category, then
so is ${\rm Pro}(\sC)$. We refer the reader to \cite[Appendix~4]{AM} for further details about
${\rm Pro}(\sC)$.

Let $\sC$ be  an abelian category and let $\{A_m\}_m$ be a pro-object
in $\sC$. We shall say that $\{A_m\}$ is bounded by an integer
$N \in \N$ if the structure map $A_{m+N} \to A_m$ is zero for all
$m \ge 0$. A pro-object $\{A_m\}$ which is bounded by an integer is classically also called
AR-zero (Artin-Rees zero), see
 \cite[Expos{\'e}~ V, Definition~ 2.2.1]{SGA5}.
 We shall say that $\{A_m\}$ is bounded by $\infty$ if
$\{A_m\} = 0$ in ${\rm Pro}(\sC)$.

Let $X$ be a Noetherian scheme. By a sheaf (or pre-sheaf) of pro-abelian
groups on $X$, we shall mean a pro-object in the abelian category of
sheaves (or pre-sheaves) of abelian groups on $X$. We caution the reader
that if $\{\sF_m\}$ is a sheaf of pro-abelian groups such that
the pro-abelian group of stalks $\{\sF_{m,x}\}$ is zero for all
$x \in X$, then we can not in general conclude that $\{\sF_m\}$
is zero. However, the following is still true and will be used 
repeatedly in this article.

\begin{lem}\label{lem:pro-sheaves}
Let $\{\sF_m\}$ be a sheaf of pro-abelian groups on a Noetherian scheme
$X$. Suppose that for every integer $m \ge 0$, there is an integer
$N(m) \ge 0$ such that the structure map
$\sF_{m + N(m),x} \to \sF_{m,x}$ is zero for all $x \in X$.
Then $\left\{\sF_{m}\right\} = 0$. If 
there is an integer $N \ge 0$ 
such that $\{\sF_{m,x}\}$ is bounded by $N$ for all $x \in X$,
then $\{\sF_m\}$ is bounded by $N$. 

If $f \colon \{\sF_m\} \to \{\sF'_m\}$ is an isomorphism of
sheaves of pro-abelian groups on $X$, then the morphism of
pro-abelian groups $H^i(f) \colon \{H^i(X,\sF_m)\} \to  \{H^i(X,\sF'_m)\}$
is an isomorphism for all $i \ge 0$.
\end{lem}
\begin{proof}
It is elementary and is left to the reader.
\end{proof}

\subsection{The relative algebraic $K$-theory}\label{sec:Rel-K}
Given a closed immersion $D \subset X$ of schemes, we let $K(X,D)$ be the homotopy fiber of the
restriction map between the Bass-Thomason-Trobaugh non-connective algebraic 
$K$-theory spectra $K(X) \to K(D)$.
We shall let $K_i(X)$ denote the homotopy groups of $K(X)$ for $i \in \Z$.
We similarly define $K_i(X,D)$. The canonical maps
of spectra $K(X, (m+1)D) \to K(X,mD)$ together give rise to a 
pro-spectrum $\{K(X,mD)\}$ and hence a pro-abelian group
$\{K_i(X,mD)\}$ for every $i \in \Z$.

If $X = \Spec(R)$ is affine and $D = V(I)$, we shall often write $K(X,mD)$ as $K(R, I^m)$ and
$K(X)$ as $K(R)$. For a ring $R$, we shall let $\wt{K}(R_m)$ denote the reduced
$K$-theory of $R_m$, defined as the
homotopy fiber of the augmentation map $K(R_m) \to K(R)$. 
%Observe that there exists a canonical decomposition $K(R_m) \cong \wt{K}(R_m) \times K(R)$.

We need to use a push-forward map between
the relative $K$-groups in a special situation. We describe this
below.
Let $R \to R'$ be a finite and flat extension of rings and let $(m, n)$
be a pair of integers such that $m \ge n \ge 0$. 
Let $f \colon \Spec(R') \to \Spec(R)$ denote
the corresponding maps between the schemes.
Since $R'_m \cong R_m \otimes_R R'$, it follows that
$R'_m \otimes_{R_m} R_n \cong R'_n$, where $R_m \surj R_n$ and 
$R'_m \surj R'_n$ are the canonical surjections. 
In particular, the diagram of schemes

\begin{equation}\label{eqn:PF-PB}
\xymatrix@C.8pc{
\Spec(R'_n) \ar[r] \ar[d] & \Spec(R_n) \ar[d] \\
\Spec(R'_m) \ar[r] & \Spec(R_m)}
\end{equation}
is Cartesian, where the vertical arrows are the closed immersions
induced by the surjections $R_m \surj R_n$ and $R'_m \surj R'_n$.
Since the horizontal arrows in this diagram are flat, it follows
that $\Spec(R_n)$ and $\Spec(R'_m)$ are Tor-independent over $\Spec(R_m)$.
Since $R'_m$ is finite and flat over $R_m$, it follows from
\cite[Proposition~3.18]{TT} that ~\eqref{eqn:PF-PB} induces a
homotopy commutative diagram of spectra
\begin{equation}\label{eqn:PF-PB-0}
\xymatrix@C.8pc{
K(R'_m) \ar[r]^-{f_{m *}} \ar[d] & K(R_m) \ar[d] \\
K(R'_n) \ar[r]^-{f_{n *}} & K(R_n),}
\end{equation}
where the horizontal arrows are the push-forward and the
vertical arrows are the pull-back maps between $K$-theory spectra.
Considering the map induced between the vertical homotopy fibers,
we get a push-forward map
$f_{(m,n) *} \colon K(R'_m, {(x^{n+1})}/{(x^{m+1})}) \to 
 K(R_m, {(x^{n+1})}/{(x^{m+1})})$ between the relative
$K$-theory spectra.
The special case of the pair $(m,0)$ yields the
push-forward map 
$f_{m *} \colon \wt{K}(R'_m) \to \wt{K}(R_m)$
between the reduced $K$-theory spectra.

\begin{lem}\label{lem:proj-PF}
Let $R \to R'$ be as above and let $m \ge n \ge 0$ be two integers.
Then the diagram
\begin{equation}\label{eqn:proj-PF-0}
\xymatrix@C.8pc{
\wt{K}_i(R'_m) \ar[r]^{f_{m *}} \ar[d] & \wt{K}(R_m) \ar[d] \\
\wt{K}_i(R'_n) \ar[r]^{f_{n *}} & \wt{K}_i(R_n)}
\end{equation}
is commutative for every $i \in \Z$, 
where the vertical arrows are the pull-back maps
induced by the quotients $R_m \surj R_n$ and $R'_m \surj R_n$.
In particular, there is a push-forward map
between the pro-abelian groups
$f_* \colon \{\wt{K}_i(R'_m)\}_m \to \{\wt{K}_i(R_m)\}_m$ for every $i \in \Z$. 
\end{lem}
\begin{proof}
We fix $i \in \Z$ and consider the diagram
\begin{equation}\label{eqn:proj-PF-1}
\xymatrix@C.8pc{
\wt{K}_i(R'_m) \ar[rr]^-{f_{m *}} 
\ar[dr] \ar[dd] & & \wt{K}_i(R_m) \ar[dd] \ar[dr] & \\ 
& K_i(R'_m) \ar[dd] \ar[rr]^->>>>>{f_{m *}} & & K_i(R_m) \ar[dd] \\
\wt{K}_i(R'_n) \ar[rr]^->>>>>{f_{n *}} \ar[dr] & & \wt{K}_i(R_n) \ar[dr] & \\ 
& K_i(R'_n) \ar[rr]^-{f_{n *}} & & K_i(R_n).}
\end{equation}

We have seen in ~\eqref{eqn:PF-PB-0} that the front face of
~\eqref{eqn:proj-PF-1} is commutative. The left and the right faces 
clearly commute. The top and the bottom faces commute by
applying ~\eqref{eqn:PF-PB-0} corresponding to the pairs $(m, 0)$ and $(n,0)$,
respectively. Since the map $\wt{K}_i(R_n) \to K_i(R_n)$ is injective,
it follows that the back face commutes, as desired.
\end{proof}

Let us now assume that $k \inj k'$ is a finite extension of fields
and let $f \colon \Spec(k') \to \Spec(k)$ denote the induced
morphism of schemes. In this case,
$k[t] \inj k'[t]$ is clearly a finite and flat extension of polynomial
rings. Let $A = k[t]_{(t)}$ denote the localization of $k[t]$ at the maximal
ideal $(t)k[t] \subset k[t]$. Let $A' = k'[t]_{(t)}$ denote the
localization of $k'[t]$ at $(t)k'[t]$. We claim that $A \inj A'$ is a finite
and flat extension of discrete valuation rings.

Indeed, it is clear that there are ring extensions
$A \inj S^{-1}k'[t] \inj A'$ in which the first inclusion 
is finite and flat if we let $S = k[t] \setminus (t)$.
For every integer $m \ge 0$, we 
have a sequence of isomorphisms of $k$-algebras: 
\begin{equation}\label{eqn:elem}
S^{-1}k'[t] \otimes_A A/{(t^{m+1})} \cong (k'[t] \otimes_{k[t]} A)
\otimes_A k_m \cong k'[t] \otimes_{k[t]} k_m \hspace*{4cm} 
\end{equation}
\[
\hspace*{7cm} 
\cong
(k' \otimes_k k[t]) \otimes_{k[t]} k_m \cong k' \otimes_k k_m \cong k'_m.
\]

If we let $\fm = (t)k[t] \subset A$ and $m =0$, it follows that 
$A \inj S^{-1}k'[t]$ is a finite extension of regular semi-local integral
domains of dimension one
such that $\fm S^{-1}k'[t]$ is a maximal ideal of $S^{-1}k'[t]$.
This forces $S^{-1}k'[t]$ to be a discrete valuation ring (since $A$ is).
Since $A'$ is a also a discrete valuation ring which 
is a localization of $S^{-1}k'[t]$, different from the fraction field of
$S^{-1}k'[t]$, we must have $S^{-1}k'[t] = A'$. This proves the claim.

For any integer $m \ge 0$, we now get finite and flat ring extensions
\[
k \inj k', \ k_m \inj k'_m, \ A \inj A', \ \mbox{and} \ k(t) \inj k'(t).
\]
Each of these extensions induces a push-forward map on the
$K$-theory spectra. We denote all these push-forward maps by the
common notation $f_*$.

\begin{lem}\label{lem:push-ford}
For each $i \in \Z$, there is a commutative diagram
\begin{equation}\label{eqn:push-ford-0}
\xymatrix@C.8pc{
\wt{K}_i(k'_m) \ar@{^{(}->}[r] \ar[d]_-{f_*} & K_i(k'_m) \ar[d]^-{f_*} &
K_i(A') \ar@{^{(}->}[r] \ar[l] \ar[d]^-{f_*} & K_i(k'(t)) \ar[d]^-{f_*} \\
\wt{K}_i(k_m) \ar@{^{(}->}[r] & K_i(k_m) &
K_i(A) \ar@{^{(}->}[r] \ar[l] & K_i(k(t)),}
\end{equation}
where the horizontal arrows in the middle and the right
side squares are the natural maps on $K$-theory
induced by the ring homomorphisms. The horizontal arrows in the
left square are the canonical maps.
\end{lem}
\begin{proof}
The left square commutes by the construction of $f_*$ in
~\eqref{eqn:PF-PB-0}. The horizontal arrows in this square are split injective
via the augmentation maps. The middle square commutes by exactly the
same argument as for the commutativity of ~\eqref{eqn:PF-PB-0}
since $A'$ is finite and flat over $A$, hence $A'$ and $k_m$ are
Tor-independent over $A$. Furthermore,
$A' \otimes_A k_m \cong k'_m$ by ~\eqref{eqn:elem}. 
The square on the right side commutes by a similar reason
once we know that $k(t) \otimes_A A' \cong k'(t)$.
But this is obvious because $k'(t)$ is the field of fractions of
$A'$ on the one hand and $k(t) \otimes_A A'$ is a localization 
of the integral domain $A'$ which is finite over $k(t)$ on the
other hand. It follows that $k'(t) \subset k(t) \otimes_A A' \subset k'(t)$. 
The horizontal arrows in the right side square are injective by
the Gersten resolution of $K$-theory.
\end{proof}

\subsection{The additive higher Chow groups}\label{sec:ACG*}
We recall the definition of the higher Chow groups with modulus and
the additive higher Chow groups (see \cite{BS}, \cite{KP-1}
and \cite{KP}). Let $k$ be a field and let $X$
be an equidimensional scheme over $k$. Let $D \subset X$ be an
effective Cartier divisor. 

For any integers 
$n, q \ge 0$, we let
$\un{z}^q(X|D, n)$ denote the free abelian group on the set of integral closed 
subschemes of $X \times \square^n$ of codimension $q$ satisfying the following.
\begin{enumerate}
\item
$Z$ intersects $X \times F$ properly for each face $F \subset \square^n$.
\item
If $\ov{Z}$ is the closure of $Z$ in $X \times \ov{\square}^n$ and 
$\nu: \ov{Z}^N \to X \times \ov{\square}^n$
is the canonical map from the normalization of $\ov{Z}$, then the inequality 
(called the modulus condition)
\[
\nu^*(D \times \ov{\square}^n) \le \nu^*(X \times F^1_n)
\] 
holds in the group of Weil divisors on $\ov{Z}^N$.
\end{enumerate}

An element of the group $\un{z}^q(X|D, n)$ will be called an admissible cycle.
It is known that $\{n \mapsto \un{z}^q(X|D, n)\}$ is a 
cubical abelian group (see \cite[\S~1]{KLevine}). 
We denote this by $\un{z}^q(X|D, *)$.
We let $z^q(X|D, *) = \frac{\un{z}^q(X|D, *)}{\un{z}^q_{\rm degn}(X|D, *)}$, 
where $\un{z}^q_{\rm degn}(X|D, *)$ is the degenerate part of the cubical 
abelian group $\un{z}^q(X|D, *)$. 

The higher Chow groups with modulus of $(X,D)$ are defined as 
$\CH^q(X|D, n) = H_n(z^q(X|D, *))$.
It is clear that there is a canonical
map $\CH^q(X|(m+1)D, n) \to  \CH^q(X|mD, n)$ for every integer $m \ge 1$.
In particular, $\{\CH^q(X|mD, n)\}_{m \ge 1}$ is a pro-abelian group.

For an equidimensional scheme $X$ over $k$ and
integers $m, n \ge 0, q \ge 1$, the additive higher Chow
group of $X$ is defined by
\begin{equation}\label{eqn:ACG}
\TCH^q(X, n+1;m) := \CH^q(X \times \A^1_k| X \times (m+1)\{0\}, n).
\end{equation}
As with the Chow groups with modulus, the datum $(X, n, q)$ 
for $n, q \ge 1$ gives rise to a pro-abelian group
$\{\TCH^q(X, n;m)\}_{m \ge 0}$.

\subsection{The cycle class map}\label{sec:CCM}
In this subsection, we recall our main object of study, the cycle
class map for 0-cycles with modulus, which was constructed in \cite{GK}.
Let $X$ be a smooth quasi-projective scheme of dimension $d \ge 1$ over
a field $k$ and let $D \subset X$ be an effective Cartier divisor.
We fix an integer $n \ge 0$.

Let $z \in X \times \square^n$ be an admissible closed point
and let $f \colon z = \Spec(k(z)) \to \square^n$ be the projection map.
For $1 \le i \le n$, we let $y_i \colon \Spec(k(z)) 
\xrightarrow{f} \square^n \to \square$ be the projection to the
$i$-th factor of $\square^n$.
It follows from the face condition of $z$ that $y_i(z)$ does not meet
$0, \infty \in \square$ for any $i$. Hence, we get an element
$\un{y}(z) = \{y_1(z), \ldots , y_n(z)\}
\in K^M_n(k(z))$. Composing with the canonical map
$K^M_n(k(z)) \to K_n(k(z))$, we get an element 
$\un{y}(z) \in K_n(k(z))$. Since $\Spec(k(z)) \to X$ is finite and 
$z \notin D \times \square^n$, it follows that there is a 
push-forward map $K(k(z)) \cong K(z, \emptyset) \
\to K(X,D)$.
Letting $cyc_{X|D}([z])$ be the image of $\un{y}(z) \in K_n(k(z))$
in $K_n(X,D)$ and extending it linearly, we obtain a cycle
class map
\begin{equation}\label{eqn:CCM-gen}
cyc_{X|D} \colon z^{d+n}(X|D,n) \to K_n(X,D).
\end{equation}

The key observation in the construction of the cycle class map at
the level of the Chow group of 0-cycles with modulus is that
the composite map
$z^{d+n}(X|(n+1)D,n+1) \inj z^{d+n}(X|D,n+1) \xrightarrow{\partial}  
z^{d+n}(X|D,n)$ is zero. This yields for every $m \ge 1$, the map
\begin{equation}\label{eqn:CCM-main}
cyc_{X|mD} \colon \CH^{d+n}(X|(n+1)mD,n) \to K_n(X, mD).
\end{equation}
This coincides with the cycle class maps of Levine \cite{Levine-1}
and Totaro \cite{Totaro} when $D = \emptyset$.

It is immediate from the above construction that the maps 
$\{cyc_{X|mD}\}_{m \ge 1}$ are compatible with respect to the change in
$m \ge 1$. In particular, they together give rise to a
map of pro-abelian groups
$cyc_{X|D} \colon \{\CH^{d+n}(X|mD,n)\}_m \to \{K_n(X, mD)\}_m$.

\vskip .3cm

Applying ~\eqref{eqn:CCM-main} to the additive higher Chow groups of
the field $k$, we see that for every $m \ge 0$ and $n \ge 1$,
there is a cycle class map
$cyc_{k|m} \colon \TCH^{n}(k,n ; n(m+1)-1) \to K_{n-1}(\A^1_k, (m+1)\{0\})$.
The homotopy fiber sequence
$K(\A^1_k, (m+1)\{0\}) \to K(\A^1_k) \to K(\Spec(k_m))$ and the homotopy
invariance of $K$-theory for smooth schemes together show that the
connecting homomorphism $\partial \colon \Omega \wt{K}(\Spec(k_m)) \to 
K(\A^1_k, (m+1)\{0\})$ is a functorial weak equivalence.
Hence, we get a cycle class map
\begin{equation}\label{eqn:CCM-add}
cyc_{k|m} \colon \TCH^{n}(k,n ; n(m+1)-1) \to \wt{K}_n(k_m).
\end{equation}

The compatibility of these maps for varying values of $m \ge 0$ yields
the cycle class map at the level of pro-abelian groups
\begin{equation}\label{eqn:CCM-pro}
cyc_{k} \colon \{\TCH^{n}(k,n ; m)\}_{m} \to \{\wt{K}_n(k_m)\}_{m}
\end{equation}
for which the associated function $\lambda_n \colon \N \to \N$ 
(see \S~\ref{sec:Pro}) is given by $\lambda_n(m) = n(m+1) - 1$.

About the above cycle class map, the following were shown in \cite{GK} when the base field
has characteristic zero.

\begin{enumerate}
\item
The map $cyc_k$ of ~\eqref{eqn:CCM-pro} extends
to the additive Chow group of relative 0-cycles over 
a regular semi-local domain $R$, essentially of finite type over 
$k$.
\item
The resulting map $cyc_R$ is an isomorphism.
\end{enumerate}

In this manuscript, we wish to study this problem when $k$ has
positive characteristic.

\section{The relative Milnor $K$-groups}\label{sec:Milnor}
The relative Milnor $K$-groups were defined by Kato-Saito \cite[\S~1.3]{Kato-Saito}, Kerz \cite{Kerz10}  and
R\"ulling-Saito \cite[\S~2.7]{RS}. The groups defined by Kato-Saito and Kerz agree when all residue fields
of the underlying ring are infinite. However, they differ from the one defined by R{\"u}lling-Saito
even if all residue fields are infinite. When the underlying ring has a finite residue field, all three
are in general different from each other.
We need to establish some isomorphisms between these $K$-groups
in pro-setting
in order to prove our main results. We shall prove these isomorphisms 
in the next two sections.
%Along the way, we prove refinements of some results of Kato-Saito \cite{Kato-Saito}
%and Kerz \cite{Kerz09}.

\subsection{Kato-Saito relative Milnor $K$-groups}\label{sec:Kato-S}
For a ring $R$, the Milnor $K$-group $K^M_n(R)$ was defined by Kato \cite{Kato} to be
the $n$-th graded piece of the graded ring $K^M_*(R)$.
The latter is the quotient of the tensor algebra $T_*(R^{\times})$ by the 
two-sided graded ideal generated by the homogeneous elements
$a_1 \otimes \cdots \otimes a_n$ such that $n \ge 2$ and $a_i + a_j = 1$ for some $1 \le i \neq j \le n$.
The residue class $a_1 \otimes \cdots \otimes a_n \in T_n(R^{\times})$
in $K^M_n(R)$ is denoted by the Milnor symbol $\un{a} = \{a_1,  \ldots , a_n\}$.
If $a_i + a_j = 1$ for some $1 \le i \neq j \le n$, we shall usually say that $\un{a}$ is a Kato symbol
(or Kato relation).
If $I \subset R$ is an ideal, the relative Milnor $K$-group
$K^M_n(R,I)$ was defined by Kato-Saito \cite[\S~1.3]{Kato-Saito} as the kernel of the natural map $K^M_n(R) \to K^M_n(R/I)$.
In order to give a simple description of $K^M_*(R,I)$, we need the following elementary step.

\begin{lem}\label{lem:Milnor-elem}
  Let $R$ be a semi-local ring containing a field of cardinality at least 
three. Let $I \subset R$ be a proper ideal.
  Suppose $a \in R$ is such that $\ov{a} \in (R/I)^{\times}$. We can then find an
  element $b \in I$ such that $a+b \in R^{\times}$. If $\ov{a}, 1 - \ov{a} \in (R/I)^{\times}$, then
  we have $a+b, 1 -(a+b) \in R^{\times}$.
\end{lem}
\begin{proof}
  Let $\{\fm_1, \ldots , \fm_r\}$ be the set of maximal ideals of $R$ which contain $I$ and
  let $\{\fn_1, \ldots , \fn_s\}$ be the set of remaining maximal ideals of $R$.
  Since $R$ contains a field of cardinality at least 
three, we can find an element $u \in R$ such that $u, 1-u \in R^{\times}$.
  By the Chinese remainder theorem, we can find an element $b \in I$ such that
  $b \equiv u - a$ modulo $\fn_i$ for $1 \le i \le s$.

  If $\ov{a} \in (R/I)^{\times}$, then we must have that $a \notin \fm_i$ for any $i$.
  It follows that $a + b \notin \fm_i$ for all $i$. Since $a+ b$ is a unit modulo $\fn_j$ for each $j$,
  $a+b$ can not belong to $\fn_j$ either. It follows that $a + b \in R^{\times}$.
  Suppose now that $\ov{a}, 1 - \ov{a} \in (R/I)^{\times}$. Then $ 1 - (a+b)$ can not be in any $\fm_i$.
  On the other hand, we have $1 - (a + b) \equiv 1 - u$, which is a unit modulo $\fn_j$ for every $j$ and hence
  $1 - (a+b)$ can not be in any $\fn_j$ either. It follows that $a+b, 1 - (a+b) \in R^{\times}$.
 \end{proof}

The next lemma is due to Kato-Saito (see \cite[Lemma~1.3.1]{Kato-Saito})
when $R$ is local. We shall need a version of this also
for the relative $K$-theory of Kerz \cite{Kerz09} 
(see \lemref{lem:Kerz-symbol}).

\begin{lem}\label{lem:Semi-local}$($\cite[Lemma~1.3.1]{Kato-Saito}$)$
  Let $R$ be a semi-local ring and $I \subset R$ a proper ideal. Then $K^M_n(R) \to K^M_n(R/I)$
  is surjective. If $R$ contains a field of cardinality at least three, then $K^M_n(R,I)$ is generated by the
  Milnor symbols $\{a_1, \ldots , a_n\}$ such that $a_i \in {\rm Ker}(R^{\times} \to (R/I)^{\times})$
for some $1 \le i \le n$.
\end{lem}
\begin{proof}
  The first part of the lemma follows from \lemref{lem:Milnor-elem}. The reader can check from
  the proof of \lemref{lem:Milnor-elem} that this part does not require $R$ to contain
  a field (take $u =1$).
  To prove the second part, let $N$ be the ideal of $T_*(R)$ generated Kato relations (see above) and
  the Milnor symbols of the type given in the statement of the lemma. It is clear that
  the map ${T_*(R)}/N \to K^M_*(R/I)$ is surjective. It suffices therefore to construct a
  map $\eta_I \colon K^M_*(R/I) \to {T_*(R)}/N$ such that the composite
  ${T_*(R)}/N \to K^M_*(R/I) \to {T_*(R)}/N$ is identity.

  Given $a'_1, \ldots , a'_n \in (R/I)^{\times}$, we can use the first part of the lemma to find
  $a_1, \ldots , a_n \in R^{\times}$ such that $\ov{a_i} = a'_i$ for each $i$.
  We let $\eta_I(a'_1 \otimes \ldots \otimes a'_n) = \{a_1, \ldots , a_n\}$ modulo $N$.
  To show that this does not depend on the choice of the lifts, we first let $n =2$
  (note that the $n =1$ case is clear).
  We let $a_i, b_i \in R^{\times}$ be such that $a_ib^{-1}_i \equiv 1$ modulo $I$ for $i =1,2$.
  We then have the identities $\{a_1, a_2\} = \{a_1, a_2b^{-1}_2\} + \{a_1, b_2\}$ and
  $\{b_1, b_2\} = \{a^{-1}_1b_1, b_2\} + \{a_1, b_2\}$. The $n =2$ case follows immediately from
  these two identities.

  Suppose now that $n \ge 3$ and
  we are given $a_i, b_i \in R^{\times}$ such that $a_ib^{-1}_i \equiv 1$ modulo $I$
  for $1 \le i \le n$. We then have the identities
  \[
\{a_1, \ldots , a_n\} = \{a_1, \ldots , a_{n-1}, a_nb^{-1}_n\} + \{a_1, \ldots a_{n-2}\} \cdot \{a_{n-1}, b_n\}
\ \mbox{and}
\]
\[
\{b_1, \ldots , b_n\} = \{b_1, \ldots , a^{-1}_{n-1}b_{n-1}, b_n\} + \{b_1, \ldots , b_{n-2}\} \cdot
\{a_{n-1}, b_n\}.
\]

Using the induction and the above two identities, we conclude the proof of well-definedness of $\eta_I$.
It is easy to check that $\eta_I$ is multi-linear and hence defines a ring homomorphism
$\eta_I \colon T_*(R/I) \to {T_*(R)}/N$. Furthermore, it follows from \lemref{lem:Milnor-elem}
that $\eta_I$ kills Kato relations. In particular, we get a ring homomorphism
$\eta_I \colon K^M_*(R/I) \to {T_*(R)}/N$. It is clear that the composite map
${T_*(R)}/N \to K^M_*(R/I) \to {T_*(R)}/N$ is identity. This finishes the proof.
\end{proof}

\begin{remk}\label{remk:Only-local}
  Lemmas~\ref{lem:Milnor-elem} and \ref{lem:Semi-local} hold even if $R$ does not contain a field
  as long as ${R}/{\fm}$ contains at least three elements for every maximal ideal
  $\fm \subset R$ not containing $I$. In particular, they hold if $R$ is local.
  It is however not clear that they hold for all semi-local rings. The problem lies in lifting Kato relations from $R/I$ to $R$.
  \end{remk}

  \subsection{Kerz's relative Milnor $K$-groups}\label{sec:Kerz*}
  In \cite{Kerz09}, Kerz defined the Milnor $K$-group $K^{MS}_n(R)$ for a ring $R$ to be the $n$-th graded piece of the
  graded ring $K^{MS}_*(R)$. The latter is the quotient of the tensor algebra $T_*(R^{\times})$ by the two-sided
  graded ideal generated by the Steinberg symbols $a \otimes (1- a) \in T_2(R^{\times})$, where
  $a,  1-a \in R^{\times}$.  This is the direct extension of $K$-theory of fields defined by Milnor
  \cite{Milnor-1} to rings.
  If $I \subset R$ is an ideal, we let
  $K^{MS}_n(R,I)$ be the kernel of the natural map $K^{MS}_n(R) \to K^{MS}_n(R/I)$.
A straightforward imitation of the proof of \lemref{lem:Semi-local} shows the following.

\begin{lem}\label{lem:Kerz-symbol}
  \lemref{lem:Semi-local} is valid for the map $K^{MS}_*(R) \to K^{MS}(R/I)$ and the group $K^{MS}_*(R,I)$.
\end{lem}

It is evident from the above definitions and Lemmas~\ref{lem:Semi-local} and ~\ref{lem:Kerz-symbol}
that there are natural surjections
\begin{equation}\label{eqn:Kato-Kerz}
  K^{MS}_*(R) \surj K^M_*(R) \ \mbox{and} \ K^{MS}_*(R,I) \surj K^M_*(R,I),
\end{equation}
where the second surjectivity holds under the assumption that $R$ is a semi-local ring containing a field of cardinality at least three.
It follows from \cite[Lemma~2.2]{Kerz09} that the maps of ~\eqref{eqn:Kato-Kerz} are isomorphisms if
  $R$ is a semi-local ring with infinite residue fields. In fact, the reader can easily check that the
  proof of \cite[Lemma~2.2]{Kerz09} remains valid if $R$ is a local ring whose residue field contains at least five
  elements (if $a = 1 - b$ with $b \notin R^{\times}$, take $s_1 = s_2 = 2- b$ in Kerz's proof).
  We therefore get:

  \begin{lem}\label{lem:surj-Milnor}
    Let $R$ be either a semi-local ring with infinite residue fields or
    a local ring with residue field having cardinality at least five.  Let $I \subset R$ be
    a proper ideal. Then the maps $K^{MS}_*(R) \to K^M_*(R)$ and $K^{MS}_*(R,I) \to K^M_*(R,I)$ are isomorphisms.
    \end{lem}

\subsection{Gabber-Kerz improved relative Milnor $K$-groups}\label{sec:Improved-K}
When the residue fields of $R$ are not all infinite, 
then the Milnor $K$-theories $K^M_*(R)$ and $K^{MS}_*(R)$ do
not have good properties. 
For example, the Gersten conjecture does not hold for them even if $R$ is a regular
local ring containing a field.
If $R$ is a finite product of local rings containing a field, 
Gabber (unpublished) and Kerz \cite{Kerz10} defined an improved 
version of Milnor $K$-theory,
which is denoted as $\wh{K}^M_*(R)$. This is a graded commutative ring and
\cite[Proposition~10 (3), Theorem~13]{Kerz10} imply that 
there are natural maps of graded commutative rings 

\begin{equation}\label{eqn:Relation}
K^{MS}_*(R) \surj K^M_*(R) \stackrel{\eta_R}{\surj} \wh{K}^M_*(R) \xrightarrow{\psi_R} K_*(R). 
\end{equation}

The first two arrows are isomorphisms 
if $R$ is a field. Moreover, the Gersten resolution holds for $\wh{K}^M_*(R)$
if $R$ is regular. 
Given an ideal $I \subset R$, one defines $\wh{K}^M_*(R, I)$
similar to $K^M_*(R,I)$. The product structures on the 
(improved) Milnor and Quillen $K$-theories yield 
natural graded homomorphisms of $K^{MS}_*(R)$-modules
\begin{equation}\label{eqn:Relation-0}
K^{MS}_*(R,I) \surj K^M_*(R, I) \xrightarrow{\eta_{R|I}} \wh{K}^M_*(R,I) \xrightarrow{\psi_R} \wh{K}_*(R,I), 
\end{equation}
where the latter is the kernel of the canonical map $K_*(R) \to K_*(R/I)$.

If $k$ is a field and $X$ is a scheme over $k$, let $\sK^M_{n,X}$ be the Zariski sheaf on $X$ whose stalk
at a point $x$ is $K^M_{n}(\sO_{X,x})$ for $n \ge 0$. In \cite{Kerz10},
Kerz actually shows that there is a Zariski sheaf $\wh{\sK}^M_{n,X}$ on $X$ with a natural surjective map
$\sK^M_{n,X} \to \wh{\sK}^M_{n,X}$ such that the stalk of $\wh{\sK}^M_{n,X}$ at $x$ is
$\wh{K}^M_n(\sO_{X,x})$.
 If $X = \Spec(A)$, we let $\wh{K}^M_n(A) = H^0(X, \wh{\sK}^M_{n,X})$.
If $I \subset A$ is an ideal, we let $\wh{K}^M_n(A,I)$ be the kernel of the canonical map
$\wh{K}^M_n(A) \to \wh{K}^M_n(A/I)$.

\begin{defn}\label{defn:Image-rel}
  For $n \ge 0$, we let $\ov{K^M_n(A,I)}$ denote the image of the canonical map
  $K^M_n(A,I) \to \wh{K}^M_n(A,I)$. We similarly define $\ov{K^M_n(A)}$ to be the image of the  canonical map $K^M_n(A) \to \wh{K}^M_n(A)$
  \end{defn}

\begin{rem}
  Observe that if $A$ is a local ring, then $\ov{K^M_n(A)} = \wh{K}^M_n(A) $ but
  it may happen that $ \ov{K^M_n(A,I)} \neq \wh{K}^M_n(A,I)$ if $I \neq 0$ is
  a proper ideal in $A$.
  Moreover, for a semi-local ring $A$, we may have that $\ov{K^M_n(A)} \neq \wh{K}^M_n(A) $.
\end{rem}

If $R$ is a regular semi-local ring containing $k$ and if $F$ is its 
total ring of quotients, then the Gersten complex
\cite{Kato}
\begin{equation}\label{eqn:Gersten}
  0 \to \wh{K}^M_n(R) \to K^M_n(F) \to {\underset{{\rm ht}(\fp) = 1}\oplus} K^M_{n-1}(k(\fp)) \to
  \cdots \to {\underset{{\rm ht}(\fp) = n}\oplus} K^M_{0}(k(\fp))
  \end{equation}
  coincides with the Gersten complex for higher Chow groups except at the first place. Bloch \cite{Bloch-1}
  showed that the latter is exact when $R$ is local. In general, it follows from the above definition of $\wh{K}^M_n(R)$
  that ~\eqref{eqn:Gersten} is exact at the first two terms. 

We refer to \cite[\S~3]{GK} for more details about other properties of the improved Milnor
$K$-theory. In this paper, we shall always use the improved
Milnor $K$-theory of rings, whenever it is defined.
For fields, we shall use the notations of Milnor and
improved Milnor $K$-groups interchangeably.

\subsection{R\"ulling-Saito relative Milnor $K$-groups}
\label{sec:RSM}
Let $R$ be a local domain with fraction field $F$ and let $I = (f)$ be a 
principal ideal, where $f \in R$ is a non-zero divisor. 
Let $R_f$ denote the localization $R[f^{-1}]$ obtained by
inverting the powers of $f$. Let $F$ denote the ring of total quotients of
$R$ so that there are inclusions of rings $R \inj R_f \inj F$.
We let $\wh{K}^M_1(R|I) = K^M_1(R,I)$ and for $n \ge 2$, we let
$\wh{K}^M_n(R|I)$ denote the image of the canonical map
of abelian groups
\begin{equation}\label{eqn:RS-0}
\wh{K}^M_1(R|I) \otimes (R_f)^{\times} \otimes \cdots \otimes (R_f)^{\times} \to \wh{K}^M_n(F),
\end{equation}
induced by the product in the Milnor $K$-theory. 
These groups were defined by  R\"ulling-Saito \cite[\S~2.7]{RS} as stalks of a sheaf (see \cite[Definition~2.4 and Lemma~2.1]{RS}).  We shall call $\wh{K}^M_*(R|I)$ the R\"ulling-Saito relative 
Milnor $K$-groups.
The basic relation between the Kato-Saito and R\"ulling-Saito relative Milnor $K$-groups
is given by the following.

\begin{lem}\label{lem:Relation-rel}
  Let $R$ be a regular local domain and let $I = (f)$ be a non-zero principal ideal.
  Then there is a commutative diagram of graded groups
  \begin{equation}\label{eqn:Relation-rel-0}
    \xymatrix@C.8pc{
      K^M_*(R,I) \ar@{^{(}->}[d] \ar@{->>}[r] & \ov{K^M_*(R,I)} \ar@{^{(}->}[r] \ar@{^{(}->}[d] &
      \wh{K}^M_*(R|I) \ar@{^{(}->}[d] \\
      K^M_*(R) \ar[r] & \wh{K}^M_*(R)  \ar@{^{(}->}[r] & \wh{K}^M_*(R_f).}
  \end{equation}
\end{lem}
\begin{proof}
 Let $F$ be the fraction field of $R$. 
We then note that the image of $K^M_*(R,I)$ in
  $ K^M_*(F)$ under the composite map
  $K^M_*(R) \surj \wh{K}^M_*(R) \inj \wh{K}^M_*(F)$ is $\ov{K^M_*(R,I)}$.
  We therefore only need to verify that the image of this composite map lies in the subgroup
  $\wh{K}^M_*(R|I) \subset \wh{K}^M_*(R_f)$. 

To prove this claim, it suffices to show that
  the canonical map $\zeta_R \colon K^M_n(R) \to K^M_n(F)$ sends $K^M_n(R,I)$ into $\wh{K}^M_n(R|I) \subset K^M_n(F)$ for
  all $n \ge 1$. We can assume $n \ge 2$ as the assertion is clear for $n =1$.

  Now, by \lemref{lem:Semi-local}, we need to show that if $\un{a} = \{a_1, \ldots , a_n\} \in K^M_n(R)$
  is such that
  $a_i \in K^M_1(R,I)$ for some $1 \le i \le n$, then $\zeta_R(\un{a}) \in \wh{K}^M_n(R|I)$.
  In other words, we have to show that as an element of $K^M_n(F)$, the symbol $\un{a}$ actually lies
  in $\wh{K}^M_n(R|I)$. However, this is immediate (see ~\eqref{eqn:RS-0})
  because the ring $K^M_*(F)$ is anti-commutative \cite[Lemma~1.1]{Milnor-1}.
  \end{proof}

  \subsection{Connection between $\wh{K}^M_*(R,I)$ and $\wh{K}^M_*(R|I)$}
  \label{sec:relation-rel-1}
  The canonical
  map $K^M_*(R,I) \to \wh{K}^M_*(R|I)$ of \lemref{lem:Relation-rel} in general
may not factor through
    $\wh{K}^M_*(R,I)$. We shall show however that this is indeed the case in some situations
    if we replace $K^M_*(R,I)$
    and $\wh{K}^M_*(R|I)$ by the pro-abelian groups $\{K^M_*(R,I^m)\}_{m}$ and
    $\{\wh{K}^M_*(R|I^m)\}_{m}$, respectively. In this subsection, we construct a map in the
    opposite direction, which is slightly easier.

    Let $R$ be a semi-local ring
     with the  maximal ideals $\fm_1, \dots, \fm_r$.
Let $R[T]$ denote the
    polynomial ring over $R$ and let $A$ denote the localization of $R[T]$
obtained by inverting all polynomials having invertible constant term. 
Then $A$ is a semi-local ring of Krull dimension $\dim(R) + 1$ and the 
maximal ideals $\fm_iA + (T)$, for $1\leq i \leq r$.
%The localization $A[T^{-1}]$ is a ring having the same Krull dimension as $R$. Unless $R$ is Artinian, the ring $A[T^{-1}]$ has infinitely many maximal ideals, including ideals $\fm_i[T^{\pm 1}]$.
    We now let $R$ be a local ring with maximal ideal $\fm$ and let $R(T)$ be the localization of $A[T^{-1}]$ at maximal ideal $\fm[T^{\pm 1}]$.    Then we have the inclusions $R[T] \inj A \inj R(T)$. The ring $R(T)$ is 
local with infinite residue  
    field $\frac{R}{\fm}(T)$. When $R$ is a field, then 
$A = R[T]_{(T)}$. If $R$ is an integral domain with fraction
    field $F$, then $A$ is an integral domain with fraction field $F(T)$.

We now fix a local integral domain $R$ containing a field. Let $A$ be the local ring
defined above. We fix the ideal $I = (T) \subset A$. Let $F$ denote the fraction field of $R$.
The key lemma to connect $\wh{K}^M_*(A|I)$ with $\wh{K}^M_*(A,I)$ is the following.

\begin{lem}\label{lem:Milnor-0}
Assume that $R$ is a regular local ring. 
Then the following inclusions hold for every pair of integers $n \ge 0, m \ge 1$.
 \begin{equation}\label{eqn:Milnor-0-0}
    \wh{K}^M_{n+1}(A|I) \subset (1 + I)\wh{K}^M_n(A) \subset \ov{K^M_{n+1}(A,I)} \subset \wh{K}^M_{n+1}(A, I).
  \end{equation}
  \begin{equation}\label{eqn:Milnor-0-1}
    \wh{K}^M_{n+1}(A|I^{m+1}) \subset (1 + I^m)\wh{K}^M_n(A) \subset \ov{K^M_{n+1}(A, I^m)}
      \subset \wh{K}^M_{n+1}(A, I^m).
  \end{equation}
\end{lem}
\begin{proof}
The map $K^M_*(A) \to \wh{K}^M_*(A)$ is surjective as $A$ is local.
This immediately implies 
  the second inclusions in ~\eqref{eqn:Milnor-0-0}
and ~\eqref{eqn:Milnor-0-1}. 
So we only
  need to show the first set of inclusions. Since $A/(T)= R$ is a regular local ring, $I$ defines an effective Cartier divisor on 
 $\Spec(A)$ which is regular (and hence simple normal crossing). The first inclusion in  ~\eqref{eqn:Milnor-0-0} now 
 follows from  \cite[Proposition~2.8(2)]{RS}. To prove ~\eqref{eqn:Milnor-0-1}, we first 
 observe that 
 $(A[T^{-1}])^{\times} = A^{\times} \cdot T^{\Z}$. Since $\{T, T\} = \{T, -1\}$ in $K_2^M(F(T))$, 
 it follows that  $\wh{K}^M_{n+1}(A|I^{m+1})$ is generated by the subgroup 
 $(1+I^{m+1}) \wh{K}^M_n(A) $ and the element of the form $\{1+aT^{m+1}, T, u_1, \dots, u_{n-1}\} \in K_{n+1}^M(F(T))$, with $a\in A$ and $u_i\in A^{\times}$. It therefore suffices to show that for $a\in A$, we have $\{1+aT^{m+1}, T\} \in (1+I^m)\cdot A^{\times} \subset K^M_2(F(T))$.   But this follows 
 from   \cite[Lemma~2.7(2)]{RS}. 
 This completes the proof of the lemma.
  \end{proof}

  \enlargethispage{10pt}
  
    Let $R$ be a regular local ring containing a field. 
We let $\wt{K}^M_n(R_m) := {\rm Ker}(\wh{K}^M_n(R_m) \surj \wh{K}^M_n(R))$.  
To apply \lemref{lem:Milnor-0}, we observe that the canonical injection
$R[T] \inj A$ induces an isomorphism $R_m \xrightarrow{\cong} {A}/{I^{m+1}}$ for all $m \ge 0$.
%Since the Kato-Saito and Gabber-Kerz Milnor
%$K$-theories for a local coincide if it contains an infinite field,
Since the maps $K^M_n({A}/{I^{m+1}}) \to \wh{K}^M_n({A}/{I^{m+1}})$ and
$K^M_n(A) \to K^M_n({A}/{I^{m+1}})$ are surjective, 
our assumption implies that 
\[
  0 \to \wh{K}^M_n(A, I^{m+1}) \to \wh{K}^M_n(A) \to \wh{K}^M_n({A}/{I^{m+1}}) \to 0
\]
is an exact sequence. Using this, we therefore see 
that the canonical restriction map $\wh{K}^M_n(A) \surj \wh{K}^M_n(R_m)$
induces a natural (in $R$) isomorphism
\begin{equation}\label{eqn:Milnor-1}
\theta_{R_m} \colon \frac{\wh{K}^M_n(A, I)}{\wh{K}^M_n(A, I^{m+1})}
  \xrightarrow{\cong} \wt{K}^M_n(R_m).
  \end{equation}
  
  Since the map $R \to {A}/I$ is an isomorphism, the quotient maps
  $K^M_n(A) \surj K^M_n({A}/{I})$ and $\wh{K}^M_n(A) \surj \wh{K}^M_n(A/I)$ compatibly split via the augmentation.
  This implies that the induced map on the kernels $K^M_n(A,I) \to \wh{K}^M_n(A,I)$ is also surjective.
  In particular, the map
  \begin{equation}\label{eqn:Milnor-1-0}
    \frac{\ov{K^M_n(A,I)}}{\ov{K^M_n(A, I^{m+1})}} \to \frac{\wh{K}^M_n(A, I)}{\wh{K}^M_n(A, I^{m+1})}
  \end{equation}
  is surjective.

Recall from \lemref{lem:Relation-rel} that $\ov{K^M_n(A, I^{m})}$ is a 
subgroup of  $\wh{K}^M_n(A|I^{m})$ for every $m \ge 1$. 
The main result we wished to prove in this section is the following.

 \begin{prop}\label{prop:Milnor-reln-main}
 Let $R$, $A$ and $F$ be as in \lemref{lem:Milnor-0} and let $n \ge 0$ be an integer.
Then the kernel (resp. cokernel) of the natural morphism of pro-abelian groups
$\{\ov{K^M_n(A, I^{m})}\}_m \to \{\wh{K}^M_n(A|I^m)\}_m$
is bounded by 0 (resp. 1).
In particular, the inclusions $\wh{K}^M_n(A|I^m) \inj \wh{K}^M_n(F(T))$ 
induce a morphism of pro-abelian groups
   \begin{equation}\label{eqn:Milnor-reln-main-0}
     \left\{\frac{\wh{K}^M_n(A|I)}{\wh{K}^M_n(A|I^m)}\right\}_{m}
     \to  \left\{\frac{\wh{K}^M_n(A, I)}{\wh{K}^M_n(A, I^{m})}\right\}_{m}
     \end{equation}
with $\lambda(m) = m+1$ whose cokernel is bounded by zero.
\end{prop}
   \begin{proof}
     For $n = 0$, all the groups are zero. For $n \ge 1$, the first assertion is a direct
     consequence of Lemmas~\ref{lem:Relation-rel} and
~\ref{lem:Milnor-0}. Since the map $\ov{K^M_n(A,I)} \to
     \wh{K}^M_n(A|I)$ is an isomorphism, also by \lemref{lem:Milnor-0}, 
it follows that the kernel of the map
   \[
     \left\{\frac{\ov{K^M_n(A,I)}}{\ov{K^M_n(A, I^{m})}}\right\}_{m} \to
     \left\{\frac{\wh{K}^M_n(A|I)}{\wh{K}^M_n(A|I^m)}\right\}_{m}
   \]
   is bounded by $1$ and the cokernel is bounded by zero.
Equivalently, 
$\left\{\frac{\wh{K}^M_n(A|I^m)}{\ov{K^M_n(A, I^{m})}}\right\}_{m}$
is bounded by $1$. 
By combining this with ~\eqref{eqn:Milnor-1-0}, we see that
the inclusions $\wh{K}^M_n(A|I^m) \inj \wh{K}^M_n(A[T^{-1}])$ 
induce a morphism of pro-abelian groups such as in
~\eqref{eqn:Milnor-reln-main-0} with $\lambda(m) = m+1$.
It is clear that the map 
$\frac{\wh{K}^M_n(A|I)}{\wh{K}^M_n(A|I^{m+1})}
     \to \frac{\wh{K}^M_n(A, I)}{\wh{K}^M_n(A, I^{m})}$
is surjective for each $m \ge 1$ since 
$\wh{K}^M_n(A|I) = \ov{K^M_n(A,I)} = \wh{K}^M_n(A, I)$.
This finishes the proof.
\end{proof}

\section{The de Rham-Witt complex and $K$-theory}\label{sec:DWC}
  Proposition~\ref{prop:Milnor-reln-main} is not quite enough to prove our main results.
  We need the map ~\eqref{eqn:Milnor-reln-main-0} to be actually an isomorphism. We shall
  prove this stronger assertion in this section using the de Rham-Witt complex. We shall use this
  isomorphism in \S~\ref{sec:CCM-MK} to obtain our cycle class map to the 
relative Milnor $K$-theory of truncated polynomial rings.
  We shall also use the de Rham-Witt complex to prove some more results on 
Milnor and Quillen $K$-groups of truncated polynomial rings in \S~\ref{sec:MQ}. 

We shall not recall the definition of the de Rham-Witt complex here.
  Instead, we refer the reader to \cite{Hesselholt-Ac} and \cite[\S~1]{R} for its definition
  and basic properties. We only recall that for a regular semi-local ring $R$ and integer $m \ge 1$,
  there are natural isomorphisms of abelian groups
  \begin{equation}\label{eqn:Witt-vec}
    \gamma_{R,m} \colon \W_m(R) \xrightarrow{\cong} \frac{(1 + TR[[T]])^{\times}}
    {(1 + T^{m+1}R[[T]])^{\times}} \xrightarrow{\cong} \wt{K}^M_1(R_m);
   \end{equation}
   \[
     \gamma_{R,m}(\un{a}) = \stackrel{m}{\underset{i = 1}\prod} (1 - a_iT^i).
     \]

     We shall often write $\gamma_{R,m}(\un{a}) = \gamma_{R,m}((a_1 \ldots , a_m))$ as
     $\gamma(\un{a})$ if the context of its usage is clear. For any $a \in R$, we recall
     that $[a] = (a, 0, \ldots , 0) \in \W_m(R)$ denotes the Teichm\"uller lift of $a$.  Note that $\gamma_{R,m}$ is clearly natural in $R$ and $m \ge 1$.
     The following lemma is a direct consequence of 
\cite[Proposition~2.3]{KP-2}.
     We state it separately as we will need it a few times in our proofs.

     \begin{lem}\label{lem:DWC-inj}
       Let $R$ be a regular semi-local ring containing a field and let $F$ be 
its total ring of quotients.
       Then the canonical map $\W_m\Omega^n_R \to \W_m\Omega^n_F$ is injective for all
       $m \ge 1$ and $n \ge 0$.
       \end{lem}

\subsection{Generators of de Rham-Witt complex}\label{sec:DRW*}
We give a generating set of the de Rham-Witt complex of semi-local rings
in this subsection.
After proving this result, we realized that we only need it when the
underlying ring is a field (see the proofs of Lemmas~\ref{lem:Key-lem-1} and
~\ref{lem:Key-lem-2}) to prove the main results of this paper. 
In this special case, a presentation 
of de Rham-Witt complex is already known
by Hyodo-Kato \cite{HK} and R{\"u}lling-Saito \cite{RS}.
We however present this generalization here since it has independent interest.
For instance, it is indispensable for the proof of
\cite[Corollary~6.4]{KP-2}.

\begin{prop}\label{prop:HK-RS}
Let $R$ be a regular semi-local ring containing a field $k$ of characteristic 
$p > 0$ and let 
$m \ge 1, n \ge 0$ be two integers. Assume that either $k$ is infinite 
or $R$ is local. Then the map
\begin{equation}\label{eqn:HK-main-0}
(\W_m(R) \otimes \bigwedge^n_{\Z}R^{\times}) \oplus (\W_m(R) \otimes 
\bigwedge^{n-1}_{\Z} R^{\times}) \to \W_m\Omega^n_R,
\end{equation}
defined by
\[
w \otimes (a_1 \wedge \cdots \wedge a_n) \mapsto w\dlog[a_1] \cdots \dlog[a_n]
\ \mbox{and}
\]
\[
w \otimes (a_1 \wedge \cdots \wedge a_{n-1}) \mapsto dw 
\dlog[a_1] \cdots \dlog[a_{n-1}],
\] 
is surjective. 
\end{prop}
\begin{proof}
When $R$ is a local ring, this result was proven (with a description of the kernel of this map)
by Hyodo-Kato 
\cite[Proposition~4.6]{HK} in the $p$-typical case and the reduction of
the general case to the $p$-typical case was shown by R{\"u}lling-Saito
\cite[Proposition~4.4]{RS}. The new assertion is that 
the surjectivity part of the result of Hyodo-Kato holds for
semi-local rings as well.
 
It suffices to prove the proposition for the $p$-typical de Rham-Witt
complex. This reduction is routine (see the proof of 
\cite[Proposition~4.4]{RS}) and requires no condition on $R$.
We shall use the notations of $p$-typical de Rham-Witt complex
$W_m\Omega^n_R$, where $W_m(-) := \W_{\{1, p, \ldots , p^{m-1}\}}(-)$.

We let $E^n_{R,m}$ denote the abelian group on the left hand side of
~\eqref{eqn:HK-main-0}. 
We need to show that the map $\theta^n_{R,m} \colon
E^n_{R,m} \to W_m\Omega^n_R$ is surjective.
We shall fix $n \ge 0$ and prove that $\theta^n_{R,m}$ is a surjection 
by induction on $m \ge 1$.

If $m = 1$, then we know that $W_m\Omega^n_R \cong \Omega^n_R$.
In this case, it follows from our assumption and
\cite[Lemma~7.4]{GK} that  $R$ is additively generated by its units.
The latter immediately implies the desired surjectivity.

In the general case, we know by \cite[I.3.15.2, p.~576]{Illusie} that
there is an exact sequence
\[
0 \to {\rm gr}^{m}_C W\Omega^n_R \to W_{m+1}\Omega^n_R \to W_{m}\Omega^n_R
\to 0,
\]
where ${\rm gr}^{\bullet}_C W\Omega^n_R$ is the associated graded module
for the canonical filtration on the de Rham-Witt complex of $R$.
One knows that the canonical filtration of $W_{m+1}\Omega^n_R$ coincides 
with its $V$-filtration (see \cite[Lemma~3.2.4]{HM-1}). Equivalently, we have
${\rm gr}^{m}_C W\Omega^n_R = {\rm gr}^{m}_V W\Omega^n_R := 
V^mW_{1}\Omega^n_R + d V^mW_{1}\Omega^{n-1}_R$.

We now consider the commutative diagram with exact 
rows
\begin{equation}\label{eqn:HK-RS-1}
\xymatrix@C.8pc{
& {R \otimes (\bigwedge^n_{\Z}R^{\times} \oplus \bigwedge^{n-1}_{\Z}R^{\times})}
\ar[r]^-{V_m \otimes {id}} \ar[d] & E^n_{R, m+1} \ar[r] \ar[d] &
E^n_{R,m} \ar[r] \ar[d] & 0 \\
0 \ar[r] & {\rm gr}^{m}_V W\Omega^n_R \ar[r] & W_{m+1}\Omega^n_R \ar[r] &
W_{m}\Omega^n_R \ar[r] &  0.}
\end{equation}

The $m =1$ case of the proposition and the expression of
${\rm gr}^{m}_V W\Omega^n_R$ given above show that the
left vertical arrow in ~\eqref{eqn:HK-RS-1} is surjective.
The right vertical arrow is surjective by induction. It follows that
the middle vertical arrow is surjective, as desired.
\end{proof}

\subsection{Milnor $K$-theory and de Rham-Witt complex}
\label{sec:DWC-M}
Let $X$ be a Noetherian scheme. 
For an integer $m \ge 0$, we let $X_m = X \times \Spec(k_m)$.
We let $\wh{\sK}^M_{n, m, X}$ 
be the Zariski sheaf on $X$ whose stalk at a point
$x \in X$ is the Milnor $K$-group $\wh{K}^M_n({\sO_{X,x}[t]}/{(t^{m+1})})$.
We let $\wt{K}^M_{n,m, X}$ be the kernel of the split surjection
$\wh{\sK}^M_{n, m, X} \surj \wh{\sK}^M_{n, X}$.
We let ${\sK}_{n, m, X}$ be the Zariski sheaf on $X$ whose stalk at a point
$x \in X$ is the Quillen $K$-group ${K}_n({\sO_{X,x}[t]}/{(t^{m+1})})$.
We define $\wt{\sK}_{n,m,X}$ just as we defined $\wt{K}^M_{n,m, X}$.

For a point $x \in X$, we shall denote the
ring $A(\sO_{X,x})$ (which is obtained exactly as in
\S~\ref{sec:relation-rel-1}, where $R$ is replaced by $\sO_{X,x}$)
in short by $A_x$. Note that $A_x$ is local.
We let $\wh{\sK}^M_{n, X|I^m}$ denote the Zariski sheaf on $X$ whose
stalk at a point $x \in X$ is the group $\wh{K}^M_n(A_x|I^m)$.
We let $\wh{K}^M_{n, (X,I^m)}$ denote the Zariski sheaf on $X$ whose
stalk at a point $x \in X$ is the group $\wh{K}^M_n(A_x, I^m)$.

We now fix a regular semi-local ring $R$ containing a field $k$ of
characteristic $p > 0$ and let $X = \Spec(R)$.
We let $A := A(R)$ be the semi-local ring
defined in \S~\ref{sec:relation-rel-1} and let $I = (T) \subset A$.
We shall continue to follow the notations of
\S~\ref{sec:Milnor}. 
We shall use the following result of  R{\"u}lling-Saito
which gives an explicit relation between the
de Rham-Witt complex and the R{\"u}lling-Saito relative Milnor $K$-theory.

\begin{thm}\label{thm:RS-main}$($\cite[Theorem~4.8]{RS}$)$
Let $n \ge 0$ and $m \ge 1$ be two integers. Then the map
of sheaves
\begin{equation}\label{eqn:RS-main-0}
\lambda_X \colon \W_m\Omega^n_X \to 
\frac{\wh{\sK}^M_{n+1, X|I}}{\wh{\sK}^M_{n+1, X|I^m}},
\end{equation}
which on local sections is given by
\[
  w \dlog[a_1] \cdots \dlog[a_n] \mapsto \{\gamma(w), a_1, \ldots , a_n\}
\]
and
\[
  dw \dlog[a_1] \cdots \dlog[a_{n-1}] \mapsto 
(-1)^n \{\gamma(w), a_1, \ldots ,a_{n-1}, T\},
\]
is an isomorphism.
\end{thm}

By combining \propref{prop:Milnor-reln-main}, \lemref{lem:DWC-inj}
and \thmref{thm:RS-main}, 
we now prove the following key result. This will play a crucial role in the
proofs of the main theorems. This result is also used in 
\cite[\S~6]{Morrow} without a proof when $R$ is local.

\begin{thm}\label{thm:Milnor-final}
Let $X$ be a regular scheme over a field $k$ of
characteristic $p > 0$ and let $n \ge 0$ be an integer.
Then the canonical maps of sheaves of pro-abelian groups  
\[
\left\{\frac{\wh{\sK}^M_{n, X|I}}{\wh{\sK}^M_{n, X|I^m}}\right\}_m
\to \left\{\frac{\wh{\sK}^M_{n, (X,I)}}{\wh{\sK}^M_{n, (X,I^m)}}\right\}_m
\to \left\{\wt{\sK}^M_{n, m, X}\right\}_m
\]
are isomorphisms.
\end{thm}
\begin{proof}
The theorem is obvious for $n = 1$ from the definitions of
various groups. We thus assume that $n \ge 2$.
Note here that the first map is induced by ~\eqref{eqn:Milnor-0-0} and ~\eqref{eqn:Milnor-0-1}
and second by ~\eqref{eqn:Milnor-1}. Indeed, note that all the sheaves 
$\wh{\sK}^M_{n, X|I}, \wh{\sK}^M_{n, X|I^m}, \wh{\sK}^M_{n, (X,I)}$ and 
$\wh{\sK}^M_{n, (X,I^m)}$ are sub-sheaves of the constant sheaf $K_n(F(T))$, where 
$F$ is the total ring of quotients of $R$. It then follows by ~\eqref{eqn:Milnor-0-0} and ~\eqref{eqn:Milnor-0-1} that 
$\wh{\sK}^M_{n, X|I} = \wh{\sK}^M_{n, (X,I)}$ and 
$\wh{\sK}^M_{n, X|I^m} \subset \wh{\sK}^M_{n, (X,I^{m-n})}$. A similar argument yields the second map using ~\eqref{eqn:Milnor-1}
so that it is a level-wise isomorphism.
We only need to show that the first map is an isomorphism too.
By \lemref{lem:pro-sheaves} and \propref{prop:Milnor-reln-main},
we only need to show that if $R$ is local, then the kernel of
the level-wise surjective map
\begin{equation}\label{eqn:Milnor-final-0}
\left\{\frac{\wh{K}^M_n(A|I)}{\wh{K}^M_n(A|I^m)}\right\}_{m}
     \surj \left\{\frac{\wh{K}^M_n(A, I)}{\wh{K}^M_n(A, I^{m})}\right\}_{m}
\end{equation}
is bounded by an integer not depending on $R$.

We first assume that $R$ is a field.
If $R$ is a finite field, then it is well known that
$\W_m\Omega^{n-1}_R = 0$. It follows from \propref{prop:Milnor-reln-main} and
\thmref{thm:RS-main} that both sides of ~\eqref{eqn:Milnor-final-0}
are bounded by zero. 
If $R$ is an infinite field, then $A$ contains an infinite field. Hence, the map
$\ov{K^M_n(A, I^m)} \to \wh{K}^M_n(A,I^m)$ is an isomorphism (see \cite[Proposition~10 (5)]{Kerz10}). We can therefore replace $\wh{K}^M_n(A,I^m)$ by
$\ov{K^M_n(A, I^m)}$ in ~\eqref{eqn:Milnor-final-0}.
We now use the maps
\[
\left\{\frac{\wh{K}^M_n(A|I)}{\wh{K}^M_n(A|I^m)}\right\}_{m}
     \to  \left\{\frac{\ov{K^M_n(A, I)}}{\ov{K^M_n(A, I^{m})}}\right\}_{m}
\to \left\{\frac{\wh{K}^M_n(A|I)}{\wh{K}^M_n(A|I^m)}\right\}_{m},
\]
where the second arrow is
induced by the inclusions $\ov{K^M_n(A, I^m)} \inj \wh{K}^M_n(A|I^m)$.
We now note that the function $\lambda \colon \N \to \N$ associated to the
first arrow is $\lambda(m) = m+1$ and it is identity for the second arrow.
More precisely, the composite arrow is induced by the canonical
surjections $\frac{\wh{K}^M_n(A|I)}{\wh{K}^M_n(A|I^{m+1})}
\surj \frac{\wh{K}^M_n(A|I)}{\wh{K}^M_n(A|I^m)}$.
It follows immediately that the kernel of the composite arrow
is bounded by $1$. We are therefore done.

In the general case, we let $F$ be the fraction field of $R$. Let
$B = F[T]_{(T)} = A(F)$ and $J = IB \subset B$. It is clear that the diagram
\begin{equation}\label{eqn:Milnor-final-1}
  \xymatrix@C.8pc{
    \{\W_m\Omega^{n-1}_R\}_m \ar[r]^-{\cong} \ar[d] &
    \left\{\frac{\wh{K}^M_{n}(A|I)}{\wh{K}^M_{n}(A|I^{m})}\right\}_m \ar[r] \ar[d] &
    \left\{\frac{\wh{K}^M_n(A, I)}{\wh{K}^M_n(A, I^{m})}\right\}_{m} \ar[d] \\
  \{\W_m\Omega^{n-1}_F\}_m \ar[r]^-{\cong} &
    \left\{\frac{\wh{K}^M_{n}(B|J)}{\wh{K}^M_{n}(B|J^{m})}\right\}_m 
\ar[r] &
    \left\{\frac{\wh{K}^M_n(B, J)}{\wh{K}^M_n(B, J^{m})}\right\}_{m}}
  \end{equation}
is commutative, where the vertical arrows are the canonical base change maps.

It follows from \lemref{lem:DWC-inj} that the left vertical arrow
in ~\eqref{eqn:Milnor-final-1} is level-wise
injective. Since $R$ is local, it follows from \thmref{thm:RS-main} that the 
left horizontal arrows on the top and the bottom are level-wise
isomorphisms. It follows that the middle vertical arrow is
level-wise injective. We showed above that the kernel of the
second horizontal arrow on the bottom is bounded by $1$. 
We deduce that the kernel of the
second horizontal arrow on the top must also be bounded by $1$. 
This proves the theorem.
\end{proof}

Combining \thmref{thm:Milnor-final} and ~\eqref{eqn:Milnor-1}, 
we get the following.

\begin{cor}\label{cor:Milnor-truncated}
Let $R$ be a regular semi-local ring containing a field $k$ of
characteristic $p > 0$ and let $X= \Spec(R)$.  Let $n \ge 0$ be an
integer. 
Then there are isomorphisms of pro-abelian groups
  \[
    \{\W_m\Omega^{n-1}_R\}_m \stackrel{\lambda_R} 
{\underset{\cong}\longrightarrow} 
\left\{H^0(X, {\wh{\sK}^M_{n, X|I}}/{\wh{\sK}^M_{n, X|I^m}})\right\}_m
\stackrel{\theta_R}{\underset{\cong}\longrightarrow} 
\left\{\wt{K}^M_n(R_m)\right\}_m,
  \]
which are natural in $R$. 
\end{cor}
\begin{proof}
Since $U \mapsto \W_m\Omega^{n-1}_{\sO(U)}$ is a Zariski sheaf on $X$
(see \cite[Proposition~I.1.13.1]{Illusie}),
the first isomorphism follows from \thmref{thm:RS-main}.
We now show the second isomorphism.

\thmref{thm:Milnor-final} yields a map
$\{\W_m\Omega^{n-1}_R\}_m \xrightarrow{\theta_R \circ \lambda_R} 
\left\{H^0(X, \wt{K}^M_{n, m, X})\right\}_m$
whose kernel and cokernel are bounded by $1$.
It suffices therefore to show that 
$\wt{K}^M_n(R_m) \cong H^0(X, \wt{K}^M_{n, m, X})$.
Using the augmentation $R_m \surj R$, it enough to show that
$\wh{K}^M_n(R_m) \cong H^0(X, \wh{\sK}^M_{n, m, X})$.
But this is clear from the definition of the improved Milnor
$K$-theory given in \S~\ref{sec:Improved-K} once we observe that
$\wh{\sK}^M_{n, m, X}$ is nothing but the direct image sheaf
$\pi_*(\wh{\sK}^M_{n, X_m})$, where $\pi \colon 
X_m \to X$ is the projection.
Indeed, we have
$H^0(X, \wh{\sK}^M_{n, m, X}) = H^0(X, \pi_*(\wh{\sK}^M_{n, X_m}))
= H^0(X_m, \wh{\sK}^M_{n, X_m}) = \wh{K}^M_n(R_m)$.
\end{proof}

If $R$ is local, the above result is equivalent to the 
following.

\begin{cor}\label{cor:Milnor-truncated-local}
Let $R$ be as in \corref{cor:Milnor-truncated} and $n \ge 0$ an
integer. Assume $R$ is local.
Then there are isomorphisms of pro-abelian groups
  \[
    \{\W_m\Omega^{n-1}_R\}_m \stackrel{\lambda_R} 
{\underset{\cong}\longrightarrow} 
\left\{\frac{\wh{K}^M_n(A|I)}{\wh{K}^M_n(A|I^m)}\right\}_m
\stackrel{\theta_R}{\underset{\cong}\longrightarrow} 
\left\{\wt{K}^M_n(R_m)\right\}_m,
  \]
which are natural in $R$.
\end{cor}
 
 In particular, we have the following corollary stating that the improved Milnor $K$-groups of the truncated polynomials over finite fields are pro-zero. 
 
 \begin{cor}\label{cor:Milnor-trucated-finite-fields}
 Let $k$ be a finite field and $n\geq 2$. Then the pro-group $\{\wh{K}_n^M(k_m) \}_m $ is zero. 
 \end{cor}
 \begin{proof}
 It follows from \corref{cor:Milnor-truncated-local}, using the fact that 
the Milnor $K$-groups $K^M_n(k)$ are zero for $n\geq 2$.
 \end{proof}

\begin{remk}\label{remk:bound}
It should be noted that we actually showed in
Corollaries~\ref{cor:Milnor-truncated} and ~\ref{cor:Milnor-truncated-local}
that $\lambda_R$ is a level-wise isomorphism while the
kernel and cokernel of $\theta_R$ are bounded by $1$.
\end{remk}

\begin{remk}\label{remk:BDI}
  Let $R$ be as in \corref{cor:Milnor-truncated-local}, where we assume
  further that it is essentially of finite type over a perfect field and
  $\dim(R) < p$. In this special case, one can obtain a direct proof
  of \corref{cor:Milnor-truncated-local} as follows.

  We consider the maps
  \[
    \{\W_m\Omega^{n-1}_R\}_m \stackrel{\theta_R \circ \lambda_R}{\surj}
    \left\{\wt{K}^M_n(R_m)\right\}_m \surj
    \left\{\wt{K}^{\rm sym}_n(R_m)\right\}_m,
\]
where $\wt{K}^{\rm sym}_n(R_m)$ is the reduced symbolic $K$-theory used by
Bloch \cite{Bloch-IHES}. We observed in \propref{prop:Milnor-reln-main} that
the first arrow is surjective. It must therefore be an isomorphism because
the composite arrow is an isomorphism by \cite[Th{\`e}or{\'e}me~I.5.2]{Illusie}.
We warn the reader however that \thmref{thm:Milnor-final} does not follow
from the local case since we can not work with stalks in order to
show that a given morphism between pro-sheaves is an isomorphism.
\end{remk}

\section{Milnor vs Quillen relative $K$-theories}\label{sec:MQ}
In this section,  we shall use \thmref{thm:Milnor-final} to
prove some relations between the relative Milnor and Quillen $K$-groups
for regular semi-local rings (see \propref{prop:M-Q-pro}). As an 
immediate consequence, we shall prove our main results
sans \thmref{thm:Main-0}(4).

\subsection{Milnor to Quillen 
$K$-theory}\label{sec:Map}
Let $R$ be a regular ring containing a field $k$ and let $X = 
\Spec(R)$. Let $m, n \ge 0$ be two integers. Recall the
sheaves $\wt{\sK}^M_{n,m, X}$ and $\wt{\sK}_{n,m, X}$ from \S~\ref{sec:DWC-M}.
There is a canonical map $\wt{K}_n(R_m) \to H^0(X, \wt{\sK}_{n,m, X})$ which is
functorial in $R$ and $m \ge 0$. We shall need to know that this
map is close to being an isomorphism in order to construct a 
map from the relative Milnor to Quillen $K$-theory. To 
prove a precise statement, we need another result.

Suppose that ${\rm char}(k) = p > 0$.
By the main result of \cite{Hesselholt-tower}, there is a map of
pro-abelian groups
$\left\{\W_m\Omega^{n-1}_R\right\}_m \to \left\{\wt{K}_n(R_m)\right\}_m$
which is natural in $R$. In particular,
there is a map of sheaves of pro-abelian groups
$\left\{\W_m\Omega^{n-1}_X\right\}_m \to
\left\{\wt{\sK}_{n,m, X}\right\}_m$.
This map has the following property.

\begin{lem}\label{lem:Hesselholt-T}
The maps $\left\{\W_m\Omega^{n-1}_R\right\}_m \to 
\left\{\wt{K}_n(R_m)\right\}_m$ and
$\left\{\W_m\Omega^{n-1}_X\right\}_m \to
\left\{\wt{\sK}_{n,m, X}\right\}_m$ are isomorphisms.
\end{lem}
\begin{proof}
Let $\sE^n_m$ and $\sF^n_m$ denote the kernel and
cokernel of the map $\W_m\Omega^{n-1}_X \to \wt{\sK}_{n,m, X}$, respectively.
By \lemref{lem:pro-sheaves}, it suffices to show that 
for every $m \ge 0$, there are integers $N(m)$ and $N'(m)$ such that
the maps of stalks $\sE^n_{m+ N(m),x} \to \sE^n_{m,x}$ and
$\sF^n_{m + N'(m), x} \to \sF^n_{m,x}$ are zero for all $x \in X$.
But this follows directly from 
\cite[Theorem~A, Theorem~6.3 (iii)]{Hesselholt-tower}.
In fact, one can take $N(m) = 1$ for all $m$ and $N'(m)$ depends
only on $m$ and $p$. The identical proof works for 
the map $\left\{\W_m\Omega^{n-1}_R\right\}_m \to 
\left\{\wt{K}_n(R_m)\right\}_m$ too because Hesselholt's result holds for
$R$ as well (using N{\'e}ron-Popescu desingularization).
\end{proof}

Fix $n \ge 0$. We can now prove:

\begin{lem}\label{lem:Hesselholt-app}
The map $\wt{K}_n(R_m) \to H^0(X, \wt{\sK}_{n,m, X})$ is an isomorphism if
${\rm char}(k) = 0$. The map of
pro-abelian groups
$\left\{\wt{K}_n(R_m)\right\}_m \to \left\{H^0(X, \wt{\sK}_{n,m, X})\right\}_m$
is an isomorphism if ${\rm char}(k) > 0$.
\end{lem}
\begin{proof}
Assume first that ${\rm char}(k) = 0$.
In this case, it follows from \cite[Theorem~10]{Hesselholt-Hand}
that $\wt{K}_n(A_m) \cong {\underset{i \ge 1}\oplus} 
(\Omega^{n+1-2i}_A)^m$ for any regular ring $A$ containing $k$.
In particular, $\wt{\sK}_{n, m, X}$ is a quasi-coherent sheaf on $X$ defined 
by the $R$-module $\wt{K}_n(R_m)$.
This immediately implies the desired result.

Suppose now that ${\rm char}(k) > 0$ and consider the
commutative diagram of pro-abelian groups
\begin{equation}\label{eqn:Hesselholt-app-0}
\xymatrix@C.8pc{
\left\{\W_m\Omega^{n-1}_R\right\}_m \ar[d] \ar[r] &
\left\{\wt{K}_n(R_m)\right\}_m \ar[d] \\
\left\{H^0(X, \W_m\Omega^{n-1}_X)\right\}_m \ar[r] &
\left\{H^0(X, \wt{\sK}_{n,m, X})\right\}_m.}
\end{equation}

The left vertical arrow is an isomorphism by 
\cite[Proposition~I.1.13.1]{Illusie} and the usual
$p$-typical decomposition argument. The top horizontal arrow
is an isomorphism by \lemref{lem:Hesselholt-T}. The bottom horizontal
arrow is an isomorphism by Lemmas~\ref{lem:pro-sheaves} and
~\ref{lem:Hesselholt-T}. We conclude that the right vertical arrow 
is an isomorphism too.
\end{proof}

Since the Quillen $K$-theory sheaf on the spectrum of a regular
semi-local ring containing a field is acyclic by Quillen's
Gersten resolution, \lemref{lem:Hesselholt-app} implies the following.

\begin{cor}\label{cor:Hesselholt-app*}
Let $R$ be a regular semi-local ring containing a field $k$. Then
the map ${K}_n(R_m) \to H^0(X, {\sK}_{n, m, X})$ is an isomorphism if
${\rm char}(k) = 0$. The map of
pro-abelian groups
$\left\{{K}_n(R_m)\right\}_m \to \left\{H^0(X, {\sK}_{n,m, X})\right\}_m$
is an isomorphism if ${\rm char}(k) > 0$.
\end{cor}

Recall that unless $R$ is local, it is not known if there exists
a canonical map from the Milnor $K$-theory defined by Gabber and Kerz
to the Quillen $K$-theory. We can however now show using the
previous results that such a map
exists for the truncated polynomials in the pro-setting.

\begin{cor}\label{cor:Milnor-Quillen-map}
Let $R$ be a regular semi-local ring containing a field $k$ and let
$n \ge 0$ be an integer. Then there is a map of
pro-abelian groups
$\left\{\wh{K}^M_n(R_m)\right\}_m \to \left\{{K}_n(R_m)\right\}_m$
which is natural in $R$. The same holds for the relative $K$-groups.
\end{cor}
\begin{proof}
We let $X = \Spec(R)$. 
At any rate, we have a natural map of sheaves
$\wh{\sK}^M_{n,m, X} \to {\sK}_{n,m, X}$ by the main result of
\cite{Kerz10}. This gives rise to a 
commutative diagram
\begin{equation}\label{eqn:Milnor-Quillen-map-0}
\xymatrix@C.8pc{
\left\{\wh{K}^M_n(R_m)\right\}_m \ar@{.>}[r] \ar[d]_{\cong} &
\left\{{K}_n(R_m)\right\}_m \ar[d]^{\cong} \\
\left\{H^0(X, \wh{\sK}^M_{n,m, X})\right\}_m
\ar[r] & \left\{H^0(X, {\sK}_{n,m, X})\right\}_m.}
\end{equation}

The left vertical arrow is a level-wise isomorphism by definition
and the right vertical arrow is an isomorphism by
\corref{cor:Hesselholt-app*}. The corollary follows.
\end{proof}

\subsection{Identification of relative Milnor and Quillen 
$K$-theory}\label{sec:Inj*}
We shall now show that the canonical map from relative Milnor to
Quillen $K$-theory that we constructed
in \S~\ref{sec:Map} is an isomorphism. We need the following
to prove its injectivity.

\begin{lem}\label{lem:Torsion-0}
Let $X$ be a Noetherian regular scheme over a field of
characteristic $p > 0$. Let $q \ge 1$ be an integer. Then 
$\left\{_q \W_m\Omega^n_X\right\}_m = 0$ for all $n \ge 0$.
\end{lem}
\begin{proof}
It suffices to show that if $R$ is a regular semi-local ring 
containing a field of characteristic $p > 0$, and $m \ge 1$ is an
integer, then the map $_q \W_{m'}\Omega^n_R \to \W_{m}\Omega^n_R$ is zero
for some integer $m' \ge m$, depending only on $m$ and $n$.  

Using \lemref{lem:DWC-inj}, it suffices to prove this assertion
for fields. So we let $k$ be a field
  of characteristic $p > 0$. Write $q = p^rs$, where $p \nmid s$. It is then clear that
  $_q \W_m\Omega^n_k = \ _{p^r}\W_m\Omega^n_k$. We therefore need to show that given
  any integer $m \ge 1$, the map $_{p^r}\W_{m'}\Omega^n_k \to \W_m\Omega^n_k$ is zero for
  all $m' \gg m$.

  Using the $p$-typical decomposition of $\W_m\Omega^n_k$ and the fact that this
  decomposition is finite and is compatible with the restriction maps
  $\W_{m'}\Omega^n_k \to \W_m\Omega^n_k$,
  %(see \cite[\S~1, p.~4]{Hesselholt-tower}
  it suffices to prove the last assertion
for the $p$-typical de Rham-Witt forms $W_m\Omega^n_k$.
  We thus have to show that given $m \ge 1$, the map
  $_{p^r}W_{m'}\Omega^n_k \to W_m\Omega^n_k$ is zero for
  all $m' \gg m$, depending only on $m$ and $n$. 
However, this is an immediate consequence of
  a theorem of Illusie (see \cite[Proposition~I.3.4, p.~569]{Illusie}, 
see also \cite[Lemma~2.3]{R})
  % see also \cite[\S~2.3]{Morrow}
  that the canonical and $p$-filtrations (and also the $V$-filtration) of 
$W_r\Omega^n_k$ coincide.
\end{proof}

\begin{lem}\label{lem:Milnor-Quillen-inj}
Let $X$ be a Noetherian regular scheme over a field $k$ of
characteristic $p > 0$. Let $n \ge 0$ be an integer. Then the canonical map
$\left\{\wt{\sK}^M_{n, m, X}\right\}_m 
\to \left\{\wt{\sK}_{n, m, X}\right\}_m$ is injective.
\end{lem}
\begin{proof}
For $n \le 1$, the lemma is obvious. So we assume $n \ge 2$.
We fix an integer $m \ge 1$.
It suffices to show that if $R$ is regular local ring containing $k$
and $F^n_m$ is the kernel of the map $\wt{K}^M_n(R_m) \to \wt{K}_n(R_m)$,
then there exists an integer $m' \gg m$, depending only on $m$ and $n$,
such that the map $F^n_{m'} \to F^n_m$ is zero.

It follows from \lemref{lem:DWC-inj} and \corref{cor:Milnor-truncated} that the
kernel of the 
  map $\left\{\wt{K}^M_n(R_m)\right\}_m \to \left\{\wt{K}^M_n(F_m)\right\}_m$ 
is bounded by $1$.
  Using the commutativity of this map with the similar map the 
between Quillen $K$-groups,
  it suffices therefore to prove our assertion for a field $k$ 
with ${\rm char}(k) > 0$. 

If $k$ is finite, then $\left\{\wt{K}^M_n(k_m)\right\}_m$ is bounded by $1$,
again
  by \corref{cor:Milnor-truncated}. 
We can therefore assume that $k$ is infinite.
  In this case, we know that $F^n_m$ is a torsion group
  of exponent $(n-1)!$ (see \cite{NS} or \cite[Proposition~10 (6)]{Kerz10}).
On the other hand, it follows from \corref{cor:Milnor-truncated} that
the map $\left\{_q\W_m\Omega^{n-1}_k\right\}_m \to 
\left\{_q\wt{K}^M_n(k_m)\right\}_m$ has kernel and cokernel bounded by
$1$ for all $q$. It suffices therefore to show that for every pair of
integers $m, q \ge 1$, there exists 
an integer $m' \gg m$, depending only on $m$ and $n$,
such that the map $_q\W_{m'}\Omega^{n-1}_k \to _q\W_{m}\Omega^{n-1}_k$ is zero.
But this is shown in the proof of \lemref{lem:Torsion-0}.
\end{proof}

The above result implies the following (see the proof of
\propref{prop:M-Q-pro}).

\begin{cor}\label{cor:M-Q-pro-inj}
Let $R$ be a regular semi-local ring containing a field of 
characteristic $p > 0$.
 Then the canonical map 
$\left\{\wt{K}^M_n(R_m)\right\}_m \to \left\{\wt{K}_n(R_m)\right\}_m$ 
(see \corref{cor:Milnor-Quillen-map})  is injective for all $n \ge 0$.
\end{cor}

\vskip.3cm

Recall that a ring $R$ containing a field of characteristic $p > 0$ is called $F$-finite, if
it is a finitely generated algebra (equivalently, a finitely generated module) over $R^p$. One knows that
$R$ is $F$-finite if it is essentially of finite type over a perfect field. This is also true
for the Henselization or completion of $R$ along any ideal.
In particular, any field which is a finitely generated over a perfect field is $F$-finite. We say that a scheme is locally $F$-finite if all its local rings are so.

\vskip .3cm

The main result of this section that we shall use later is the following.

\begin{prop}\label{prop:M-Q-pro}
  Let $R$ be an $F$-finite regular semi-local ring containing a field of 
characteristic $p > 0$.
  Then the canonical map $\left\{\wt{K}^M_n(R_m)\right\}_m \to 
\left\{\wt{K}_n(R_m)\right\}_m$ of \corref{cor:Milnor-Quillen-map} is an
  isomorphism for all $n \ge 0$.
\end{prop}
\begin{proof}
  If $R$ is local, this follows from \thmref{thm:Milnor-final} and
  \cite[Theorem~6.1]{Morrow} (which implicitly uses \thmref{thm:Milnor-final}
  or Remark~\ref{remk:BDI}).
  To prove the general case, 
we write $X = \Spec(R)$ and $X_m = \Spec(R_m)$ as before. We then have the 
strict map of sheaves of pro-abelian groups
$\psi_X \colon
\left\{\wt{\sK}^M_{n, m, X}\right\}_m \to \left\{\wt{\sK}_{n, m, X}\right\}_m$
on $X$. 
Using Lemmas~\ref{lem:pro-sheaves}, ~\ref{lem:Hesselholt-app} and 
(the proof of) \corref{cor:Milnor-truncated}, it suffices to show that
$\psi_X$ is an isomorphism. Note that $X$ is locally $F$-finite.
In view of \lemref{lem:Milnor-Quillen-inj},
we only have to show that $\psi_X$ is surjective.

We consider the commutative diagram of exact sequences of sheaves
\begin{equation}\label{eqn:MQ-pro-0}
    \xymatrix@C.8pc{
      0 \ar[r] & \wt{\sK}^M_{n, m, X} \ar[r] \ar[d]_{\psi_X} & \wh{\sK}^M_{n, m, X} 
\ar[r] \ar[d] & \wh{\sK}^M_{n,X} \ar[r] \ar[d]  & 0 \\
      0 \ar[r] & \wt{\sK}_{n, m, X} \ar[r] & {\sK}_{n, m, X} \ar[r] &
{\sK}_{n,X} \ar[r]  & 0.}
    \end{equation}

It follows by \cite[Theorem~8.1]{GL} and the exactness 
of Gersten complexes for $\wh{\sK}^M_{n,X} $ and
$\wh{\sK}_{n,X} $ (see \cite[Proposition~10(8)]{Kerz10}) that these 
sheaves have no
$p$-torsion. In particular, the two rows of ~\eqref{eqn:MQ-pro-0} remain exact 
with ${\Z}/{p^r}$-coefficients.

We first show that $\psi_X$ is surjective with ${\Z}/{p^r}$-coefficients.
Since the right vertical arrow in  ~\eqref{eqn:MQ-pro-0} is
an isomorphism with ${\Z}/{p^r}$-coefficients by \cite[Theorem~8.1]{GL},
we need to show that the middle vertical arrow is surjective 
with ${\Z}/{p^r}$-coefficients. But this follows from 
\cite[Corollary~5.5]{Morrow} (which uses the $F$-finiteness assumption).
Morrow states this corollary in the case when $X$ is the spectrum of
a local ring. However, as he explains in \cite[Remark~5.8]{Morrow}, 
the result holds at the level of sheaves too and the proof is
obtained by repeating the proof of the local ring case verbatim
and observing that the bounds in the pro-systems are controlled
while going from one point to another point of the underlying scheme. 

Indeed, Morrow shows that there are maps
\[
\left\{\wh{\sK}^M_{n, m, X}/{p^r}\right\}_m \to
\left\{{\sK}_{n, m, X}/{p^r}\right\}_m  \xrightarrow{\dlog}
\left\{W_r\Omega^n_{(X_m, X), {\rm log}}\right\}_m
\]
whose composition is level-wise surjective, where the last term
is the relative $p$-typical logarithmic de Rham-Witt complex \cite{Illusie}. 
He then shows that the second map is an isomorphism. The main argument
(one which involves the usages of pro-systems) in the proof of this 
isomorphism is the pro-HKR theorem of Dundas-Morrow \cite{DM}.
And one checks that this pro-HKR theorem holds at the level
of sheaves (see the footnote below \S~5.3 in \cite{Morrow}).  

To prove the surjectivity of $\psi_X$, we now claim
that $p^r\wt{\sK}_{n, m, X} = 0$ for some $r$, depending only on 
$m$ and $n$. 
For this, we first note that $\W_{m+n}\Omega^{n-1}_X$ is a sheaf of 
$\W_{m+n}(\F_p)$-modules. Hence, $p^{m+n}\W_{m+n}\Omega^{n-1}_X = 0$.
We next use the Hesselholt-Madsen exact sequence \cite{HM}:
\[
{\underset{i \ge 0}\oplus} \W_{(m+1)(i+1)}\Omega^{n-1-2i}_X 
\xrightarrow{\epsilon} \wt{\sK}_{n, m, X} \xrightarrow{\partial}
{\underset{i \ge 0}\oplus} \W_{i+1}\Omega^{n-2-2i}_X.
\]
We have seen above that the two end terms are annihilated by
some power of $p$ (depending only on $m$ and $n$). 
It follows that the the middle term has the same
property. This proves the claim.

We now let $\sE^n_m = {\rm Coker}(\wt{\sK}^M_{n, m, X} \to \wt{\sK}_{n, m, X})$. 
We have shown previously 
that $\left\{{\sE^n_m}/{p^r}\right\}_m = 0$ for every $r \ge 1$. 
On the other hand, for a fixed integer $m \ge 0$,  the claim
implies that $p^r\sE^n_m = 0$ for some $r \gg 0$.
It follows that the map $\sE^n_{m'} \to \sE^n_m$ is zero for all $m ' \gg m$.
In particular, $\left\{\sE^n_m\right\}_m = 0$.
This shows that $\psi_X$ is surjective and finishes the 
proof of the proposition.
\end{proof}

\subsection{The cycle class map to Milnor $K$-theory}\label{sec:CCM-MK}
 It was shown by R\"ulling \cite{R} (for fields) and Krishna-Park \cite{KP-2} (for
semi-local rings) that the additive
higher Chow groups $\TCH^n(R,n;m)$ for $m, n \ge 1$ together form the universal restricted
Witt-complex (see \cite[\S~1]{R} for definition) over $R$. In particular, there is an
isomorphism of restricted Witt-complexes
\begin{equation}\label{eqn:Witt-C}
  \tau_R \colon \W_m \Omega^{n-1}_R \xrightarrow{\cong} \TCH^n(R,n;m)
\end{equation}
for every $m,n \ge 1$. This map is given by
\begin{equation}\label{eqn:Witt-C-0}
  \tau_R(w \dlog[a_1] \cdots \dlog[a_{n-1}]) = V(\gamma(w), y_1 - a_1, \ldots , y_{n-1} - a_{n-1}),
\end{equation}
where $a_i \in R^{\times}$, $\gamma(w) \in R[T]$ is the polynomial defined in \eqref{eqn:Witt-vec} and
$V(I)$ denotes the closed subscheme of $\Spec(R) \times \square^{n-1} \cong
\Spec(R[T, y_1, \ldots , y_{n-1}])$, defined by the ideal $I$.
%This map for $R$ is comaptible with the analogous map for the fraction field of $R$.
We shall call $\tau_R$ by the name `the de Rham-Witt-Chow isomorphism'.
This is the additive analog of the Milnor-Chow isomorphism of \cite{EM}, \cite{NS} and
\cite{Totaro}.

Using the Chow-Witt isomorphism and \corref{cor:Milnor-truncated}, we define our cycle class map from
the additive higher Chow group of relative 0-cycles to relative Milnor $K$-theory as follows.

\begin{defn}\label{defn:CCM-Milnor-*}
  Let $R$ be as above and $n \ge 1$ an integer. 
%Assume the following.
%\begin{enumerate}
%\item
%$p > 2$.
%\item
%$k$ is either infinite or $R$ is local.
%\end{enumerate}
We define the cycle class map to Milnor $K$-theory
  to be the composite map of pro-abelian groups
  \begin{equation}\label{eqn:CM-Milnor-*-0}
    cyc^M_R = \theta_R \circ \lambda_R \circ \tau^{-1}_R \colon
    \left\{\TCH^n(R,n;m)\right\}_m \to \left\{\wt{K}^M_n(R_m)\right\}_m.
\end{equation}
\end{defn}
  
It follows from \cite[Theorem~1.1]{KP-2} that $\tau_R$ is functorial with respect to any
$k$-algebra homomorphism
between regular semi-local rings $R \to R'$ essentially of finite type over $k$.
% Need to give a proof for non-perfect field
The maps $\lambda_R$ and $\theta_R$
are clearly functorial in $R$ by their construction. It follows that $cyc^M_R$ is functorial in $R$. 
Notice also that $cyc^M_R$ is an isomorphism.

\subsection{Proofs of Theorems~\ref{thm:Main-0}((1)-(3)) and ~\ref{thm:Main-1}}
\label{sec:Proofs**}
%We now prove Theorems~\ref{thm:Main-0} (except part (4)) and ~\ref{thm:Main-1} simultaneously.
We let $k$ and $R$ be as in the these theorems. 
%To prove the theorems, we can assume that $R$ is an integral domain.
%Then $R$ satisfies the
%assumptions stated in the beginning of this section.
%The construction of the cycle class map $cyc_R$ of \thmref{thm:Main-0}, and its property (2)
%follow from \thmref{thm:Cycle-R}. The existence of the cycle class map $cyc^M_R$ of
%\thmref{thm:Main-1}, and its properties (1) and (3) are shown in \S~\ref{sec:CCM-MK}.
%The property (2) of $cyc^M_R$ in \thmref{thm:Main-1} follows from \lemref{lem:Gen-commute}.
%
%To prove property (1) of $cyc_R$ in \thmref{thm:Main-0}, note that
%$cyc^M_R$ is natural in $R$ as we showed above and $\psi_R$ is clearly natural in $R$.
%The naturality of $cyc_R$ now follows from \lemref{lem:Gen-commute},
%as this says that $cyc_R$ is same as $\psi_R \circ cyc^M_R$.
%The property (3) of the map $cyc_R$ in \thmref{thm:Main-0}
%follows immediately from Lemmas~\ref{lem:Milnor-Quillen-inj} and ~\ref{lem:Gen-commute}
%since $cyc^M_R$ is an isomorphism. The property (4) of $cyc_R$ follows directly
%from \propref{prop:M-Q-pro} by a similar argument. This finishes the proofs of
%Theorems~\ref{thm:Main-0} and ~\ref{thm:Main-1}.
We define the cycle class map as the following composition 
\begin{equation}\label{eqn:CM-Quillen-*-0}
    cyc'_R = \psi_R \circ cyc^M_R \colon
    \left\{\TCH^n(R,n;m)\right\}_m \to \left\{\wt{K}_n(R_m)\right\}_m,
\end{equation}
where the map $\psi_R$ is as in \corref{cor:Milnor-Quillen-map}.
The proofs of \thmref{thm:Main-1} and parts (1) to (3) of \thmref{thm:Main-0}
follow immediately.
$\hfill\square$

\section{The cycle class map for semi-local rings}\label{sec:CCM-R*}
The goal of the remaining two sections is to define $cyc_R$ at the level of additive higher
Chow groups and prove the final part of \thmref{thm:Main-0}.
In this section, we shall define $cyc_R$ which generalizes the construction of
~\eqref{eqn:CCM-pro} from fields to regular semi-local rings over a field.
In the next section, we shall show the agreement between $cyc_R$ and $cyc'_R$ under our
assumptions.

We fix a field $k$ of characteristic $p > 0$ and let $R$ be a regular semi-local ring
which is essentially of finite type over $k$. We let $F$ denote the total ring of quotients of $R$.
Let $\Sigma$ denote the set of maximal ideals of $R$.
Recall our function $\lambda \colon \Z_+ \to \Z_+$ given by $\lambda(m) = n(m+1) - 1$ in ~\eqref{eqn:CCM-pro}.

\subsection{A pro-Gersten for $K$-theory}\label{Sec:DWC-K}
  In \cite[Theorem~10.2]{GK} (see its proof),
  it was shown that if $R$ contains $\Q$,
  the base change map $\wt{K}_n(R_m) \to \wt{K}_n(F_m)$ is injective for all $m, n \ge 0$,
 where $F$ is the total ring of quotients of $R$. However, we do not know if this inclusion holds in
  positive characteristic. We shall use a result of Hesselholt-Madsen \cite{HM}
 to prove the following partial result which will imply the validity of this inclusion in the pro-setting. We shall need this result in order to
construct our cycle class map.

 \begin{lem}\label{lem:K-inj}
   Let $n \ge 0$ be any integer and let
$e \ge 1$ be an integer not divisible by $p$.
   Then the base change map $\eta_{R,e} \colon 
\wt{K}_n(R_{e-1}) \to \wt{K}_n(F_{e-1})$ is injective. 
In particular, for every $m \ge 1$, the canonical map
${\rm Ker}(\wt{K}_n(R_{mp}) \to \wt{K}_n(F_{mp}))
\to \wt{K}_n(R_{m})$ is zero.
 \end{lem}
 \begin{proof}
   We only have to show that $\eta_{R,e}$ is injective as the second assertion of the lemma
   immediately follows from this.
We can assume that $R$ is an integral domain so that $F$ is a
field.
   %Suppose we show the first assertion of the lemma.
   %To show the desired injectivity of pro-abelian groups, we fix $n \ge 0$ and
   %let $E_m = {\rm Ker}(\wt{K}_n(R_m) \to \wt{K}_n(F_m))$. It is then clear that
   %the map $E_{m'} \to E_m$ is zero for every integer $m' \ge m$ such that $p \nmid m'$.
   %It follows that $\{E_m\}_m = 0$. We therefore have to show that $\eta_{R,e}$ is injective.
We now fix an integer $n \ge 0$. It was shown by Hesselholt-Madsen \cite{HM} that there is a
   natural exact sequence
   \begin{equation}\label{eqn:K-inj-0}
     {\underset{i \ge 0}\oplus} \W_{i+1}\Omega^{n-2i}_R \xrightarrow{V_e}
     {\underset{i \ge 0}\oplus} \W_{e(i+1)}\Omega^{n-2i}_R \xrightarrow{\epsilon}
     \wt{K}_{n+1}(R_{e-1}) \xrightarrow{\partial} {\underset{i \ge 0}\oplus} \W_{i+1}\Omega^{n-1-2i}_R,
     \end{equation}
     where $V_e$ is the Verschiebung map.
 By comparing this exact sequence with the analogous exact sequence for $F$ and using
 \lemref{lem:DWC-inj}, the proof of the injectivity of $\eta_{R,e}$ reduces to showing
 that for every $n \ge 0$ and $m \ge 1$, the square
 \begin{equation}\label{eqn:K-inj-1} 
   \xymatrix@C.8pc{
     \W_m\Omega^n_R \ar@{^{(}->}[d] \ar[r]^-{V_e} & \W_{me}\Omega^n_R \ar@{^{(}->}[d] \\
     \W_m\Omega^n_F \ar[r]^-{V_e} & \W_{me}\Omega^n_F}
 \end{equation}
 is Cartesian.

 To show this, let $\alpha \in \W_{me}\Omega^n_R$ and
 $\beta \in \W_m\Omega^n_F$ be such that $\alpha = V_e(\beta) \in  \W_{me}\Omega^n_F$. We consider the commutative diagram
 \[
 \xymatrix@C.8pc{
   \W_m\Omega^n_R \ar@{^{(}->}[d] \ar[r]^-{V_e} & \W_{me}\Omega^n_R \ar@{^{(}->}[d]
   \ar[r]^-{F_e} & \W_m\Omega^n_R \ar@{^{(}->}[d] \\
   \W_m\Omega^n_F \ar[r]^-{V_e} & \W_{me}\Omega^n_F \ar[r]^-{F_e} &
 \W_m\Omega^n_F,}
\]
where $F_e$ is the Frobenius map.
Since $F_e \circ V_e(\beta) = e \beta$ (see \cite[Definition~1.4]{R}), we get
$F_e(\alpha) = e\beta$. Since $p \nmid e$, we have that $e \in (\W_m(R))^{\times}$.
We thus get $\beta = e^{-1} F_e(\alpha) \in \W_m\Omega^n_R$.
Since the all vertical arrows in the above diagram are inclusions, it follows that
$\beta \in \W_m\Omega^n_R$ and $V_e(\beta) = \alpha$. This finishes the proof.
\end{proof}

\subsection{The sfs cycles}\label{sec:Rel-cyc}
We need the notion of sfs-cycles in order to generalize the cycle
class map of ~\eqref{eqn:CCM-pro} from fields to semi-local rings. Let $m \ge 0$ and $n \ge 1$ be two integers.
%Recall our function $\lambda \colon \Z_+ \to \Z_+$ given by $\lambda(m) = n(m+1) - 1$.
Recall  (see \S~\ref{sec:ACG*}) that $\TCH^n(R,n;m)$ is defined as the middle homology of the complex
$\Tz^n(R,n+1;m) \xrightarrow{\partial} \Tz^n(R,n;m)
\xrightarrow{\partial} \Tz^{n-1}(R,n;m)$.
Note that a cycle in $\Tz^n(R,n;m)$ has relative dimension zero over $R$. 
We shall say that an extension of regular semi-local rings 
$R \subset R'$ is simple
if there is an irreducible monic polynomial 
$f \in R[t]$ such that $R' = {R[t]}/{(f(t))}$.

Let  $X = \Spec(R)$ and  $\Sigma$ the set of all maximal ideals of R.
Let $Z \subset X \times \A^1_k \times \square^{n-1}$ be an irreducible 
admissible relative 0-cycle.
Recall from \cite[Definition~3.4]{KP-2} that 
$Z$ is called an sfs-cycle if the following hold.
\begin{enumerate}
\item
$Z$ intersects $\Sigma \times \A^1_k \times F$ properly for all faces 
$F \subset \square^{n-1}$.
\item
The projection $Z \to X$ is finite and surjective.
\item
$Z$ meets no face of $X \times \A^1_k \times  \square^{n-1}$.
\item
$Z$ is closed in $X \times \A^1_k \times \A^{n-1}_k = 
\Spec(R[t, y_1, \ldots , y_{n-1}])$ 
(by (2) above) and there is a sequence of simple extensions of
regular semi-local rings
\[
R = R_{-1} \subset R_0 \subset \cdots \subset R_{n-1} = k[Z]
\]
such that $R_0 = {R[t]}/{(f_0(t))}$ and $R_i = 
{R_{i-1}[y_i]}/{(f_i(y_i))}$ for $1 \le i \le n-1$.
\end{enumerate}
Note that an sfs-cycle has no boundary by (3) above.
We let $\Tz^n _{\sfs}(R, n;m) \subset \Tz^n(R, n;m)$ be the 
subgroup of cycles whose irreducible components are
sfs-cycles and define
\begin{equation}\label{eqn:sfs}
\TCH^n _{\rm sfs} (R, n; m) = \frac{\Tz^n _{\sfs}(R, n;m)}
{\partial(\Tz^n (R, n+1;m)) \cap \Tz^n _{ \sfs} (R, n;m)}.
\end{equation}

It is clear that the canonical map $\TCH^n _{\rm sfs} (R, n; m) \to \TCH^n(R, n; m)$ is injective.
The following result from \cite[Theorem~1.1]{KP-3} says more.

\begin{thm}\label{thm:sfs-lemma}
The canonical map $\TCH^n _{\rm sfs} (R, n; m) \to \TCH^n(R, n; m)$ 
is an isomorphism if $k$ is infinite and perfect.
\end{thm}

\subsection{The cycle class map to Quillen $K$-theory}\label{sec:CCM-R-0}
The construction of the map $cyc_R$ for $\TCH^n _{\rm sfs} (R, n; m)$ is obtained by word by 
word repetition of the construction of the cycle class map for fields described in
\S~\ref{sec:CCM}. 
So let $Z \subset X \times \A^1_k \times  \square^{n-1}$ be an 
irreducible sfs-cycle and let $R' = k[Z]$.
Let $f \colon Z \to X \times \A^1_k$ be the projection map.
Let $g_i \colon Z \to \square$ denote the $i$-th projection.
Then the sfs property implies that each $g_i$ defines an element of 
${R'}^{\times}$ and this in turn
gives a unique element $
%cyc^M_{R'}([Z]) = 
\{g_1, \ldots , g_{n-1}\} 
\in K^M_{n-1}(R')$.
We let $cyc_{R'}([Z])$ be its image in $K_{n-1}(R')$ under the map
$K^M_{n-1}(R') \to K_{n-1}(R')$.
Since $Z$ does not meet $X \times \{0\}$, we see that the finite map $f$ 
defines a push-forward map of spectra $f_* \colon K(R') \to K(R[t], (t^{m+1}))$.
We let $cyc_R([Z]) = f_*(cyc_{R'}([Z])) \in K_{n-1}(R[t], (t^{m+1}))$.
We extend this definition linearly to get a cycle map $cyc_R
\colon  \Tz^n _{ \sfs} (R, n;m) \to K_{n-1}(R[t], (t^{m+1}))$.

\begin{lem}\label{lem:CCM-R-0}
The assignment $[Z] \mapsto cyc_R([Z])$ defines a cycle class map
\[
cyc_{R} \colon \TCH^n_{\rm sfs}(R, n; \lambda(pm)) \to K_{n-1}(R[t], (t^{m+1}))
\]
which is functorial for the inclusion $R \inj F$.
\end{lem}
\begin{proof}
  Let $\pi \colon \Spec(F) \inj \Spec(R)$ be the inclusion.
  We consider the diagram
  \begin{equation}\label{eqn:CCM-R-0-0}
    \xymatrix@C.4pc{
      \partial^{-1}(\Tz^n _{ \sfs} (R, n; \lambda(pm)))
\ar[r]^-{\partial} \ar[d]_-{\pi^*} & 
\Tz^n _{ \sfs} (R, n; \lambda(pm)) \ar[d]^-{\pi^*}
\ar[r]^-{cyc_R} & K_{n-1}(R[t], (t^{pm+1})) \ar[d]^-{\pi^*} \ar[r] &
K_{n-1}(R[t], (t^{m+1})) \ar[d]^-{\pi^*} \\
      \Tz^n(F, n+1;\lambda(pm)) \ar[r]^-{\partial} & 
      \Tz^n(F, n; \lambda(pm)) \ar[r]^-{cyc_F} & K_{n-1}(F[t], (t^{pm+1})) \ar[r] &
      K_{n-1}(F[t], (t^{m+1})),}
  \end{equation}
  where the horizontal arrows in the square on the right are the structure maps of
  the pro-abelian group $\left\{K_{n-1}(R[t], (t^{m}))\right\}_{m \ge 1}$ (and for $F$) because
  $mp \ge m$. In particular, this square is commutative.
It was shown in \cite[Theorem~10.2]{GK} that all the other squares are commutative.
It follows from the case of fields (see ~\eqref{eqn:CCM-add}) that the composite map
$cyc_F \circ \partial \circ \pi^*$ is zero. We deduce from \lemref{lem:K-inj} that
the composite $cyc_R \circ \partial$ is zero. It follows that
the composition of all horizontal arrows in the top row of
~\eqref{eqn:CCM-R-0-0} is zero. This proves the lemma.
\end{proof}

Since the map $cyc_k$ is clearly functorial in $m \ge 1$, using the natural isomorphism
$\partial \colon \wt{K}_n(R_m) \xrightarrow{\cong} K_{n-1}(R[t], (t^{m+1}))$, we get
\begin{thm}\label{thm:Cycle-R}
  For every $n \ge 1$, there is a cycle class map between pro-abelian groups
  \[
    cyc_R \colon \left\{\TCH^n_{\rm sfs}(R,n;m)\right\}_{m} \to \left\{\wt{K}_n(R_m)\right\}_m
  \]
  which is functorial for the inclusion $R \inj F$ and coincides with ~\eqref{eqn:CCM-pro}
  if $R$ is a field.
\end{thm}

%\begin{remk}\label{remk:Char}
%We remark that \lemref{lem:CCM-R-0} and \thmref{thm:Cycle-R} do not make any 
%assumption on $k$ or $R$. On the other hand, our cycle class map to
%the Milnor $K$-theory exists under some extra assumptions.
%\end{remk}

\section{End of the proof of \thmref{thm:Main-0} }\label{sec:Proofs*}
We shall now complete the proof of \thmref{thm:Main-0} by proving its remaining part (4).
After \S~\ref{sec:Proofs**}, the key lemma
that remains to be proven for this purpose is \lemref{lem:Gen-commute}.
We shall prove this in few steps.
We let our ring $R$ and other notations be the same as in \S~\ref{sec:CCM-R*}.
Since $R$ is a product of integral domains and our proofs for the case
of integral domains directly generalize to products of such rings, we shall assume that
$R$ is a regular semi-local integral domain.

Hence the standing assumption of this section is
that $R$ is a regular semi-local integral domain which is
essentially of finite type over a field $k$ of characteristic $p > 0$.
We let $F$ denote the fraction field of $R$.
Recall from \corref{cor:Milnor-Quillen-map} that we have a well-defined map $\psi_R: \{\wh{K}^M_*(R_m) \}_m \to \{K_*(R_m)\}_m$. If $R$ is a field, then this map is induced by 
 the canonical map  $\psi_{R_m} \colon \wh{K}^M_*(R_m) \to K_*(R_m)$ from Milnor to Quillen $K$-theory.

\subsection{The case of fields}\label{sec:Fields*}
We first consider the case when $R$ is a field. So we let $k$ be a field of characteristic $p > 0$.
We fix an integer $n \ge 1$ and consider the diagram
\begin{equation}\label{eqn:Key-diagram}
  \xymatrix@C.8pc{
    \left\{\TCH^n(k,n;m)\right\}_m \ar[r]^-{cyc_k} \ar[d]_-{cyc^M_k} & 
    \left\{K_{n-1}(\A^1_k, (m+1)\{0\})\right\}_m \\
    \left\{\wt{K}^M_n(k_m)\right\}_m \ar[r]^-{\psi_{k_m}} & \left\{\wt{K}_n(k_m)\right\}_m
    \ar[u]_-{\partial},}
\end{equation}
where $cyc_k$ is the map of ~\eqref{eqn:CCM-pro}.
\vskip .3cm

Our goal is to show that this diagram is commutative. We shall use the shortened notation
$\psi_k$ for $\psi_{k_m}$ even if it is meant to be used for $k_m$ for different values of $m \ge 1$ in different parts of the proofs.

\begin{lem}\label{lem:Key-lem-0}
  The diagram ~\eqref{eqn:Key-diagram} is commutative for $n =1$. 
\end{lem}
\begin{proof}
It follows from ~\eqref{eqn:CCM-pro} and ~\eqref{eqn:Milnor-0-1} that all maps in
  ~\eqref{eqn:Key-diagram} are strict maps of pro-abelian groups, i.e.,
  the associated function $\lambda \colon \Z_+ \to \Z_+$ is identity
  (see \S~\ref{sec:Pro}).
  Furthermore, it was shown in the initial part of the proof of \cite[Proposition~5.1]{GK}
  that $cyc_k$ is a level-wise  isomorphism.
 It follows from ~\eqref{eqn:Witt-vec} and ~\eqref{eqn:Witt-C} that all other maps 
 are also level-wise isomorphisms. Clearly, all these are functorial in $k$.
 
Finally, to show that ~\eqref{eqn:Key-diagram} commutes level-wise for $n =1$, let
 $w \in \W_m(k)$ and let $f(T) = 1 + Tp(T) \in k[T]$ be a polynomial such that
$\gamma(w) = f(T)$ modulo $T^{m+1}$.
   The construction of the cycle class map in \S~\ref{sec:CCM} then shows that
   $cyc_k(\gamma(w))$ is the class of the finitely generated $k[T]$-module
   ${k[T]}/{(f(T))}$ in $K_0(\A^1_k, (m+1)\{0\})$ (see \cite[\S2C]{GK}).
   Since $\tau_k$ is an isomorphism, it suffices to show that this class coincides with
   $\partial(\gamma(w))$. But this follows from \cite[Lemma~2.1]{GK}.
   This proves the lemma and also proves  
  stronger versions of Theorems~\ref{thm:Main-0} and ~\ref{thm:Main-1}
  when $n =1$.
\end{proof}

\enlargethispage{20pt}

Our next goal is to prove the commutativity of ~\eqref{eqn:Key-diagram} when $n \ge 2$.
We let $n \ge 2$ and let
  $\lambda_n \colon \Z_+ \to \Z_+$ be given by $\lambda_n(m) = n(m+1) -1$.
  It is then easy to see using ~\eqref{eqn:CCM-pro} and ~\eqref{eqn:Milnor-0-1}
  that all maps in ~\eqref{eqn:Key-diagram} are morphisms of pro-abelian groups
  all of whose associated functions are same, namely, the function $\lambda_n$ above
  (note that this requires $n \ge 2$).
Moreover, for $m' \ge m$, the diagram ~\eqref{eqn:Pro-1} already commutes when $l = \lambda_n(m')$.
In the proofs below,
we shall write $\lambda_n(m)$ simply as $\lambda(m)$ since $n$ is fixed.

To prove that the diagram ~\eqref{eqn:Key-diagram} is commutative for $k$ and $n \ge 2$,
it suffices therefore to show that for every $m \ge 1$, the square on the right in the diagram
\begin{equation}\label{eqn:Key-lem-0-1}
\xymatrix@C.8pc{
  \W_{\lambda(m)}\Omega^{n-1}_k \ar[r]^-{\tau_k} \ar[dr]_-{\theta_k \circ \lambda_k} &
  \TCH^n(k,n; \lambda(m)) \ar[r]^-{cyc_k} \ar[d]_-{cyc^M_k} & 
    K_{n-1}(\A^1_k, (m+1)\{0\}) \\
    & \wt{K}^M_n(k_m) \ar[r]^-{\psi_{k_m}} & \wt{K}_n(k_m)
    \ar[u]_-{\partial}}
\end{equation}
is commutative.
Since $\tau_k$ is an isomorphism, this is equivalent to showing that the outer
trapezium is commutative.

To show the commutativity of the trapezium, we shall use 
\propref{prop:HK-RS} for fields (due to
Hyodo-Kato \cite{HK} and R{\"u}lling-Saito \cite{RS}). 
Using this, it suffices to show that the above diagram
commutes for the generators of the two groups on the left hand side of
~\eqref{eqn:HK-main-0}.

\begin{comment}
\cite[Proposition~4.4]{RS},
which says that the map
\begin{equation}\label{eqn:Key-lem-0-2}
  (\W_{\lambda(m)}(k) \otimes \wedge^{n-1}_{\Z}(k)) \oplus (\W_{\lambda(m)}(k) \otimes \wedge^{n-2}_{\Z}(k))
  \to \W_{\lambda(m)}\Omega^{n-1}_k,
  \end{equation}

  \[
  \mbox{given \ by} \ \
  w \otimes (a_1 \wedge \cdots \wedge a_{n-1}) \mapsto w \dlog[a_1] \cdots \dlog[a_{n-1}];
  \]
  \[
    w \otimes (a_1 \wedge \cdots \wedge a_{n-2}) \mapsto dw \dlog[a_1] \cdots \dlog[a_{n-2}],
  \]
  is surjective.
  It suffices therefore to show that the desired diagram commutes for the generators on the
  two direct summands of the term on the left side of ~\eqref{eqn:Key-lem-0-2}.
\end{comment}

  We know from ~\eqref{eqn:Witt-vec} that any element $w \in \W_{\lambda(m)}(k)$ is of the form
  $\gamma^{-1}(\ov{1 - Tp(T)})$ with $p(T) \in k[T]$ and $\ov{f(T)} = f(T)$ modulo $T^{\lambda(m) +1}$.
    Since we can write $1 - Tp(T)$ as a product of irreducible
  polynomials of the form $1 - Tq(T)$, we see that $w$ is a sum of elements of the form
  $\gamma^{-1}(\ov{1 - Tp(T)})$ such that $1 - Tp(T)$ is irreducible.
  We can therefore assume that $w = \gamma^{-1}(\ov{f(T)})$, where $f(T) = 1 - Tp(T)$ is irreducible.

  In what follows below, we write $\phi_k = \theta_k \circ \lambda_k$ and $\psi_{k_m} = \psi_k$ to simplify
  the notations, where the value of $m \ge 1$ is allowed to vary.  We also write
  $\ov{f(T)} = f(t)$ in any $k_m$. We let $A = k[T]_{(T)}$.

  \begin{lem}\label{lem:Key-lem-1}
    For $n \ge 2$, we have
    \[
      \partial \circ \psi_k \circ \phi_k(w\dlog[a_1] \cdots \dlog[a_{n-1}]) =
      cyc_k' \circ \tau_k(w \dlog[a_1] \cdots \dlog[a_{n-1}]).
      \]
  \end{lem}
  \begin{proof}
    With the above notations, we have
  \begin{equation}\label{eqn:Key-lem-0-3}
  \begin{array}{lll}
    \partial \circ \psi_k \circ \phi_k(w \dlog[a_1] \cdots \dlog[a_{n-1}])
    & {=}^1 & \partial \circ \psi_k(\{\gamma(w), a_1, \ldots , a_{n-1}\}) \\
    & = & \partial \circ \psi_k(\{1- tp(t), a_1, \ldots , a_{n-1}\}) \\
    & {=}^2 &( \partial \circ \psi_k(\{1 - tp(t)\})) \cdot \psi_k(\{a_1, \ldots , a_{n-1}\}) \\
    & {=}^3 & \pi_*(1) \cdot \psi_k(\{a_1, \cdots , a_{n-1}\}) \\
    & {=}^4 & \pi_*(\{a_1, \cdots , a_{n-1}\}) \\
    & {=}^5 & cyc_k'(V(f(T), y_1 - a_1, \ldots , y_{n-1} -a_{n-1})) \\
    & {=}^6 & cyc_k' \circ \tau_k(w \dlog[a_1] \cdots \dlog[a_{n-1}]), \\
  \end{array}
\end{equation}
where $\pi \colon \Spec({k[T]}/{(f(T))}) \inj \A^1_k$ is the closed immersion.

We explain various equalities. First, $\theta_k$ being the restriction map
$\wh{K}^M_*(A, (T)) \to \wh{K}^M_*(k_m)$ (see ~\eqref{eqn:Milnor-1}), it is clear
that $\theta_k(\{f(T), a_1, \cdots , a_{n-1}\}) =
\{\gamma(w), a_1, \cdots , a_{n-1}\}$, where $f(T)$ is viewed as an element of 
$(1 + (T)) \subset A^{\times}$. The equality ${=}^1$ therefore follows from
the definition of the map $\lambda_k$ in ~\eqref{eqn:RS-main-0}.
The equality ${=}^2$ follows because $\partial$ is a $K^M_*(k)$-linear map.
The equality ${=}^3$ follows from the $n =1$ case shown in 
\lemref{lem:Key-lem-0}
and ${=}^4$ follows because
$\pi_*$ is $K^M_*(k)$-linear (see \cite[Lemma~2.2]{GK}). The equality ${=}^5$ follows
from the definition of the cycle class map in ~\eqref{eqn:CCM-gen} and ${=}^6$ follows from
~\eqref{eqn:Witt-C-0}. This finishes the proof.
\end{proof}

The final step is the following.

\begin{lem}\label{lem:Key-lem-2}
  The diagram~\eqref{eqn:Key-diagram} is commutative for $n \ge 2$.
\end{lem}
\begin{proof}
  Using the above reductions and \lemref{lem:Key-lem-1}, we only have to show that
  \begin{equation}\label{eqn:Key-lem-2-0}
    \partial \circ \psi_k \circ \phi_k(dw\dlog[a_1] \cdots \dlog[a_{n-2}]) =
      cyc_k \circ \tau_k(dw \dlog[a_1] \cdots \dlog[a_{n-2}]).
    \end{equation}
    
 We shall continue to use the above simplified notations and
 make another simplification by setting $\wt{w} = dw\dlog[a_1] \cdots \dlog[a_{n-2}]$.
 It is clear from the definition of the differential for the Witt-complex structure on the
 additive higher Chow groups (see \cite[\S~6.1]{KP-4})
 that if we let $\gamma(w) = f(t) = 1 - tp(t)$, then
 $\tau_k(dw) = d \tau_k(w)$ is the class of the cycle $V(f(T), Ty_1 - 1) \subset
 \A^1_k \times \square$
 in $\TCH^2(k,2; \lambda(m))$. As $f(T)$ is irreducible,
 $V(f(T), y_1T -1, y_2 - a_1, \ldots , y_{n-1} - a_{n-2})$ is a closed point
 $z \in \A^1_k \times \square^{n-1}$ such that $l = k(z) \cong {k[T]}/{(f(T))}$.
 We therefore have an admissible $l$-rational point
 $z_0 = V(1 - \alpha^{-1}T, y_1 - \alpha^{-1}, y_2 - a_1, \ldots  y_{n-1} - a_{n-2})$
 of $\A^1_l \times_l \square^{n-1}_l$ such that
 $[z] = \pi_*([z_0])$, where we let $\alpha = T$ modulo $(f(T))$ and
 $\pi \colon \Spec(l) \to \Spec(k)$  the projection.

 We can now write
 \[
   \begin{array}{lll}
     cyc_k \circ \tau_k(\wt{w}) & = & cyc_k([z]) \\
 & {=}^1 & \pi_* \circ cyc_l([z_0]) \\
                & = & \pi_* \circ cyc_l(V(1 - \alpha^{-1}T, y_1 - \alpha^{-1}, y_2 - a_1,
                      \ldots , y_{n-1} - a_{n-2})) \\
                & {=}^2 & \pi_* \circ \partial_l \circ \psi_l \circ
                          \phi_l(w_l \dlog[\alpha^{-1}]\dlog[a_1] \cdots \dlog[a_{n-2}]) \\
                & {=}^3 & \partial_k \circ \pi_* \circ \psi_l \circ
                          \phi_l(w_l \dlog[\alpha^{-1}]\dlog[a_1] \cdots \dlog[a_{n-2}]), \\
   \end{array}
   \]
   where $w_l = \gamma^{-1}(1 - \alpha^{-1}T) \in \W_{\lambda(m)}(l)$.
  
The equality ${=}^1$ follows from the construction of the cycle class map
   (see \cite[Lemma~4.4]{GK}), ${=}^2$ follows from \lemref{lem:Key-lem-1} for $\Spec(l)$
   and ${=}^3$ follows because the connecting homomorphism $\partial$ commutes with
   the push-forward map $\pi_*$. Note that this push-forward map exists on the
   relative $K$-theory by ~\eqref{eqn:PF-PB-0}. It suffices therefore to show that
   \begin{equation}\label{eqn:Key-lem-2-1}
     \psi_k \circ \phi_k(\wt{w}) =
     \pi_* \circ \psi_l \circ
     \phi_l(w_l \dlog[\alpha^{-1}]\dlog[a_1] \cdots \dlog[a_{n-2}]).
   \end{equation}
   
 However, we have
  \[
    \begin{array}{lll}
      \psi_k \circ \phi_k(\wt{w}) & = & \psi_k \circ \theta_k \circ \lambda_k(\wt{w}) \\
                     & = & (-1)^{n-1} \psi_k \circ \theta_k(\{\gamma(w), a_1, \ldots , a_{n-2}, T\}) \\
                     & {=}^1 & - \psi_k \circ \theta_k(\{\gamma(w), T, a_1, \ldots , a_{n-2}\}) \\
                                  & = & (- \psi_k \circ \theta_k(\{\gamma(w), T\}))
                                        \cdot \psi_k(\{a_1, \ldots , a_{n-2}\}), \\
    \end{array}
  \]
  where ${=}^1$ follows from \cite[Lemma~ 2.2]{Kerz09} as $k(T)$ is infinite.
  On the other hand, letting $\wt{w_l} = w_l \dlog[\alpha^{-1}]\dlog[a_1] \cdots \dlog[a_{n-2}]$,
  we also have
  \[
    \begin{array}{lll}
\pi_* \circ \psi_l \circ
\phi_l(\wt{w_l}) & = &
  \pi_* \circ \psi_l(\{\gamma(w_l), \alpha^{-1}, a_1, \ldots, a_{n-2}\}) \\
  & = & (\pi_* \circ \psi_l(\{\gamma(w_l), \alpha^{-1}\})) \cdot \psi_k(\{a_1, \ldots , a_{n-2}\}),
    \end{array}
  \]
  where the last equality holds by the projection formula.
  Thus, ~\eqref{eqn:Key-lem-2-1} is reduced to showing that for every $m \ge 1$,
  we have
    \begin{equation}\label{eqn:Key-lem-2-2}
      \psi_k \circ \theta_k(\{1 - Tp(T), T\}) = - \pi_* \circ \psi_l(\{\gamma(w), \alpha^{-1}\}) =
      - \pi_* \circ \psi_l(\{1 - \alpha^{-1}T, \alpha^{-1}\})
      \end{equation}
      in $\wt{K}_2(k_m)$ under the composite map
      $\frac{\wh{K}^M_2(A, (T))}{\wh{K}^M_2(A, (T^{m+1}))} \xrightarrow{\theta_k} \wt{K}^M_2(k_m)
      \xrightarrow{\psi_k} \wt{K}_2(k_m)$.
      Here, $\{1 - Tp(T), T\}$ is viewed as an element of $\wh{K}^M_2(A, (T))$
      via the inclusion
      $\wh{K}_2^M(A|(T)) \subset \wh{K}^M_2(A, (T))$ of \lemref{lem:Milnor-0}.

      The commutative diagram
      \begin{equation}\label{eqn:Key-lem-2-3}
        \xymatrix@C.8pc{
          \frac{\wh{K}^M_2(A, (T))}{\wh{K}^M_2(A, (T^{m+1}))} \ar[r]^-{\theta_k}_-{\cong} \ar[d]_-{\psi_A} &
          \wt{K}^M_2(k_m) \ar[d]^-{\psi_k} \\
          \frac{{K}_2(A, (T))}{{K}_2(A, (T^{m+1}))} \ar[r]^-{\theta_k}_-{\cong} & \wt{K}_2(k_m)}
        \end{equation}
        shows that verifying ~\eqref{eqn:Key-lem-2-2} is equivalent to showing that
\begin{equation}\label{eqn:Key-lem-2-4}
      \theta_k(\{1 - Tp(T), T\}) = - \pi_*(\{1 - \alpha^{-1}T, \alpha^{-1}\})
      \end{equation}
      in $\wt{K}_2(k_m)$ under the map
      $\frac{{K}_2(A, (T))}{{K}_2(A, (T^{m+1}))} \xrightarrow{\theta_k} \wt{K}_2(k_m)$,
      if we use the identical notation for $\{1 -Tp(T), T\} \in 
\wh{K}^M_2(A, (T))$ (resp., $\{1 - \alpha^{-1}T, \alpha^{-1}\} \in 
\wt{K}^M_2(l_m)$) and its image in $K_2(A, (T))$ via $\psi_A$ (resp.,
      in $\wt{K}_2(l_m)$ via  $\psi_l$). We shall use this convention 
      in the rest of the proof.

      To show ~\eqref{eqn:Key-lem-2-4}, we let $A' = l[T]_{(T)}$ as in the notations of
      \lemref{lem:push-ford}.
      We showed in \S~\ref{sec:Rel-K} that $A'$ is finite (and flat) over $A$ and $k_m \otimes_A A' \cong
      l_m$. Using ~\eqref{eqn:PF-PB-0}, we get push-forward maps
      $\pi_* \colon {K}_2(A', (T^i)) \to {K}_2(A, (T^i))$ for all $i \ge 0$ and a commutative diagram
 \begin{equation}\label{eqn:Key-lem-2-5}     
   \xymatrix@C.8pc{
 \frac{{K}_2(A', (T))}{{K}_2(A', (T^{m+1}))} \ar[r]^-{\theta_l}_-{\cong} \ar[d]_-{\pi_*} & \wt{K}_2(l_m)    
 \ar[d]^-{\pi_*} \\
 \frac{{K}_2(A, (T))}{{K}_2(A, (T^{m+1}))} \ar[r]^-{\theta_k}_-{\cong} & \wt{K}_2(k_m).}
\end{equation}
It suffices therefore to show that $\pi_*(\{1 - \alpha^{-1}T, \alpha^{-1}\}) = - \{1 - Tp(T), T\}$
holds under the left vertical arrow 
in ~\eqref{eqn:Key-lem-2-5}.

Since $\{1 - Tp(T), T\} \in K_2(A, (T)) \subset K_2(A)$ and  $\{1 - \alpha^{-1}T, \alpha^{-1}\} \in
K_2(A', (T)) \subset K_2(A')$ (note that these inclusions
use the splitting of $A \surj A/{(T)}$ and $A' \surj {A'}/{(T)}$), 
it suffices to show that
$\pi_*(\{1 - \alpha^{-1}T, \alpha^{-1}\}) = - \{1 - Tp(T), T\}$ in $K_2(A)$.
Using \lemref{lem:push-ford}, we further reduce to showing that this equality holds in $K_2(k(T))$ under
the push-forward map $\pi_* \colon K_2(l(T)) \to K_2(k(T))$.

But in $K_2(k(T))$, we have
\[
  \begin{array}{lll}
    - \pi_*(\{1 - \alpha^{-1}T, \alpha^{-1}\}) & = & \pi_*(\{1 - \alpha^{-1}T, T\}) \\
                                               & {=}^1 & \pi_*(\{1 - p(\alpha)T, T\}) \\
                                               & {=}^2 & N_{{l(T)}/{k(T)}}(\{1 - p(\alpha)T, T\}) \\
                                               & {=}^3 & \left\{N_{{l(T)}/{k(T)}}(1 - p(\alpha)T), T\right\} \\
                                               & {=}^4 & \{1 - p(T)T, T\}. \\
  \end{array}
\]
Here, ${=}^1$ follows because $1 - \alpha p(\alpha) = 0$ in $l$, the equality ${=}^2$ follows by the
compatibility between the Norm in Milnor $K$-theory and push-forward in Quillen $K$-theory
(see the proof of \cite[Lemma~4.4]{GK}), ${=}^3$ follows from the projection formula for norm
as $T \in k(T)^{\times}$, and ${=}^4$ is a straightforward calculation of the norm of
$1 -p(\alpha)T \in l(T)^{\times}$.
%(see the proof of \cite[Corollary~3.7]{R})
This finishes the proof of the lemma.
\end{proof}

\subsection{Back to the case of semi-local ring}\label{sec:SLC}
The following is our last key lemma before we prove the main results.
Let $cyc_R$ be the cycle class map of \thmref{thm:Cycle-R}.
Here, $R$ is the semi-local integral domain satisfying the standing
assumptions of this section.

\begin{lem}\label{lem:Gen-commute}
The diagram
  \begin{equation}\label{eqn:Gen-commute-0}
    \xymatrix@C.8pc{
      \left\{\TCH^n_{\sfs}(R,n;m)\right\}_m \ar[dr]^-{cyc_R} \ar[d]_-{cyc^M_R} & \\
      \left\{\wt{K}^M_n(R_m)\right\}_m \ar[r]^-{\psi_{R}} & \left\{\wt{K}_n(R_m)\right\}_m}
  \end{equation}
  is commutative. Equivalently, $cyc_R = cyc'_R$.
\end{lem}
\begin{proof}
  When $R$ is a field, the lemma is equivalent to the commutativity of ~\eqref{eqn:Key-diagram}.
  We now prove the general case. 
  Let $\pi \colon \Spec(F) \inj \Spec(R)$ be the inclusion of the generic point.
  We consider the diagram
  \begin{equation}\label{eqn:Gen-commute-1}
    \xymatrix@C.8pc{
      \left\{\TCH^n_{\sfs}(R,n;m)\right\}_m  \ar[rr]^-{\pi^*} \ar[dd]_-{cyc^M_R} \ar[dr]^-{cyc_R} & &
      \left\{\TCH^n(F,n;m)\right\}_m \ar[dd]^->>>>>{cyc^M_F} \ar[dr]^-{cyc_F} & \\
      & \left\{\wt{K}_n(R_m)\right\}_m \ar[rr]^->>>>>>>{\pi^*} & & \left\{\wt{K}_n(F_m)\right\}_m \\
      \left\{\wt{K}^M_n(R_m)\right\}_m \ar[rr]^-{\pi^*} \ar[ur]^-{\psi_R} & &
      \left\{\wt{K}_n(F_m)\right\}_m \ar[ur]_-{\psi_F}. &}
    \end{equation}

    We check the commutativity of various faces of
    ~\eqref{eqn:Gen-commute-1}.
    The front face clearly commutes and the back face commutes by \thmref{thm:Cycle-R}.
    The right (triangular) face commutes because $F$ is a field.
    The commutativity of the bottom face was shown in the construction of $cyc^M_R$ in
    \S~\ref{sec:CCM-MK}. A diagram chase shows that 
    $\pi^* \circ \psi_R \circ cyc^M_R = \pi^* \circ cyc_R$. We can now apply \lemref{lem:K-inj} to
    conclude that ~\eqref{eqn:Gen-commute-0} commutes. We use here an
elementary fact that if a morphism
    between two pro-abelian groups factors through the zero pro-group, then this morphism itself is zero
    (see \S~\ref{sec:Pro}).
\end{proof}

%\subsection{Proofs of Theorems~\ref{thm:Main-0} and ~\ref{thm:Main-1}}
%\label{sec:Proofs**}
%We now prove Theorems~\ref{thm:Main-0} and ~\ref{thm:Main-1} simultaneously.
%We let $k$ and $R$ be as in the these theorems. 
%To prove the theorems, we can assume that $R$ is an integral domain.
%Then $R$ satisfies the
%assumptions stated in the beginning of this section.
%The construction of the cycle class map $cyc_R$ of \thmref{thm:Main-0}, and its property (2)
%follow from \thmref{thm:Cycle-R}. The existence of the cycle class map $cyc^M_R$ of
%\thmref{thm:Main-1}, and its properties (1) and (3) are shown in \S~\ref{sec:CCM-MK}.
%The property (2) of $cyc^M_R$ in \thmref{thm:Main-1} follows from \lemref{lem:Gen-commute}.
%
%To prove property (1) of $cyc_R$ in \thmref{thm:Main-0}, note that
%$cyc^M_R$ is natural in $R$ as we showed above and $\psi_R$ is clearly natural in $R$.
%The naturality of $cyc_R$ now follows from \lemref{lem:Gen-commute},
%as this says that $cyc_R$ is same as $\psi_R \circ cyc^M_R$.
%The property (3) of the map $cyc_R$ in \thmref{thm:Main-0}
%follows immediately from Lemmas~\ref{lem:Milnor-Quillen-inj} and ~\ref{lem:Gen-commute}
%since $cyc^M_R$ is an isomorphism. The property (4) of $cyc_R$ follows directly
%from \propref{prop:M-Q-pro} by a similar argument. This finishes the proofs of
%Theorems~\ref{thm:Main-0} and ~\ref{thm:Main-1}.
%$\hfill\square$
%
%  

\vskip .4cm

\noindent\emph{Acknowledgments.}
The first author would like to thank TIFR, Mumbai for invitation in August 2019. 
Parts of this manuscript were worked out 
when the second author was visiting 
the Max Planck Institute for
Mathematics at Bonn in 2019. He thanks the institute for invitation and support.
The authors would like to thank the referee for pointing out corrections and
providing many valuable suggestions
to improve the presentation of the manuscript.

\end{document}